\documentclass[final,1p,times,authoryear]{elsarticle}

\usepackage{hyperref}

\usepackage{enumitem}
\usepackage{amsmath,amsthm,amssymb,bm}
\usepackage{multirow}
\usepackage{algpseudocode,algorithm}
\algnewcommand\algorithmicinput{\textbf{Input:}}
\algnewcommand\Input{\item[\algorithmicinput]}
\algnewcommand\algorithmicoutput{\textbf{Output:}}
\algnewcommand\Output{\item[\algorithmicoutput]}

\newcounter{savealgorithm}

\usepackage{caption}

\newtheorem{theorem}{Theorem}
\newtheorem{lemma}[theorem]{Lemma}
\newtheorem{proposition}[theorem]{Proposition}

\theoremstyle{definition}
\newtheorem{definition}[theorem]{Definition}
\newtheorem{expl}{Example}
\newtheorem{remark}{Remark}
\newtheorem{assumption}{Assumption}

\usepackage{cleveref}
\crefname{theorem}{Theorem}{Theorems}
\Crefname{theorem}{Theorem}{Theorems}
\crefname{lemma}{lemma}{lemmas}
\Crefname{lemma}{Lemma}{Lemmas}
\crefname{assumption}{Assumption}{Assumptions}
\crefname{line}{Line}{Lines}

\def\cA{{\mathcal A}}
\def\cB{{\mathcal B}}

\def\cD{{\mathcal D}}
\def\cE{{\mathcal E}}
\def\cF{{\mathcal F}}
\def\cG{{\mathcal G}}
\def\cH{{\mathcal H}}

\def\cL{{\mathcal L}}

\def\cP{{\mathcal P}}

\def\cT{{\mathcal T}}
\def\cU{{\mathcal U}}
\def\cV{{\mathcal V}}
\def\cW{{\mathcal W}}

\def\JKindep{{Jordan-Krylov independent}}
\def\JKindepce{{Jordan-Krylov independence}}
\def\JKbasis{{Jordan-Krylov basis}}

\def\JKelim{{Jordan-Krylov elimination}}
\def\KrylovGS{{Krylov generating set}}

\def\dsc{{direct sum condition}}

\newcommand{\Z}{\mathbb{Z}}
\newcommand{\K}{K}

\newcommand{\Q}{\mathbb{Q}}
\newcommand{\C}{\mathbb{C}}
\newcommand{\ann}{\mathrm{Ann}}
\newcommand{\Span}[0]{\mathrm{span}}

\newcommand{\rank}{\mathrm{rank}}
\newcommand{\lcm}{\mathrm{lcm}}
\newcommand{\lbar}{{\bar{\ell}}}
\newcommand{\ellbar}{\bar{\ell}}

\def\b{{\bm{b}}}
\def\e{{\bm{e}}}
\def\u{{\bm{u}}}
\def\v{{\bm{v}}}
\def\w{{\bm{w}}}
\def\p{{\bm{p}}}
\def\q{{\bm{q}}}
\def\r{{\bm{r}}}

\def\0{{\bm{0}}}
\def\={{\leftarrow}}

\newenvironment{acks}{\section*{Acknowledgements}}{}

\begin{document}

\begin{frontmatter}
  \title{Exact Algorithms for Computing Generalized Eigenspaces \\
  of Matrices via Jordan-Krylov Basis\tnoteref{t1}} 
  \author[niigata]{Shinichi Tajima}
  \ead{tajima@emeritus.niigata-u.ac.jp}

  \author[kanazawa]{Katsuyoshi Ohara}
  \ead{ohara@se.kanazawa-u.ac.jp}
  \ead[url]{http://air.s.kanazawa-u.ac.jp/~ohara/}

  \author[tsukuba]{Akira Terui\corref{corauthor}}
  \cortext[corauthor]{Corresponding author}
  \ead{terui@math.tsukuba.ac.jp}
  \ead[url]{https://researchmap.jp/aterui}

  \address[niigata]{Graduate School of Science and Technology, Niigata
    University, Niigata 950-2181, Japan}
  \address[kanazawa]{Faculty of Mathematics and Physics, Kanazawa
    University, Kanazawa 920-1192, Japan}
  \address[tsukuba]{Institute of Pure and Applied Sciences, University
    of Tsukuba, Tsukuba 305-8571, Japan}

  \tnotetext[t1]{This work has been partly supported by JSPS KAKENHI
    Grant Numbers JP18K03320, JP21K03291, JP20K11845, and by the
    Research Institute for Mathematical Sciences, a Joint
    Usage/Research Center located in Kyoto University.}

  \begin{abstract}
        An effective exact method is proposed for computing 
        generalized eigenspaces of a matrix of
        integers or rational numbers.
        The key of our approach is the concept of the {\JKbasis}.
        A new method, called {\JKelim}, is introduced to design an algorithm for
        computing the {\JKbasis}. 
        The resulting algorithm outputs generalized eigenspaces as a form of Jordan chains.
        Notably, in the output, components of generalized eigenvectors are expressed as 
        polynomials in the associated eigenvalue as a variable.
  \end{abstract}

\end{frontmatter}

\section{Introduction}
\label{sec:intro}

Exact linear algebra plays important roles in many fields of
mathematics and sciences. Over the last two decades, this area has been
extensively studied, and new algorithms have been
proposed for various types of computations,
such as 
computing canonical forms of matrices
(\cite{aug-cam1997},
\cite{dum-sau-vil2001},
\cite{hav-wag1999},
\cite{mor2004},
\cite{per-ste2010},
\cite{sau-zhe2004},
\cite{sto1998},
\cite{sto2001},
\cite{sto-lab1996}),
the
characteristic or the minimal polynomial of a matrix
(\cite{dum-per-wan2005}, \cite{neu-pra2008}), LU and other
decompositions and/or solving a system of linear equations
(\cite{bos-jea-sch2008},
\cite{ebe-gie-gio-sto-vil2006},
\cite{jea-per-sto2013},
\cite{may-sau-wan2007},
\cite{sau-woo-you2011}).
Also, the software has been developed
(\cite{alb2012},
\cite{che-sto2005},
\cite{linbox2002},
\cite{dum-gau-per2002},
\cite{dum2004},
\cite{dum-gio-per2008})
and comprehensive research results (\cite{gio2019})
have been presented.
\cite{kre-rob2016} 
investigate the interconnections between linear algebra and commutative algebra.
Notably, they study, in particular, generalized eigenproblems from the point of view of 
commutative algebra and its applications to polynomial ideals.

In the context of symbolic computation for linear algebra,
we proposed algorithms for eigenproblems
including computation of
spectral decomposition and 
(generalized) eigendecomposition
(\cite{oha-taj2009}, \cite{taj2013}, \cite{taj-oha-ter2014}, \cite{taj-oha-ter2018}).  
We studied eigenproblems from the point of view of spectral decomposition and proposed an exact 
algorithm for computing eigenvectors of matrices over integer or rational numbers by using minimal 
annihilating polynomials (\cite{taj-oha-ter2018b}).

In this paper, we treat generalized eigenspaces 
via Jordan chains as an extension of our previous papers.
In a naive method for computing generalized eigenspaces,
the generalized eigenvector is
obtained by solving a system of linear equations.
However, in general, the method has the disadvantage of algebraic number arithmetic.
It is often time-consuming for matrices of dozens of dimensions.
As methods for computing eigenvectors without solving a system of linear equations,
\cite{tak-yok1990} have proposed one 
by using the Frobenius normal form of a matrix and
\cite{mor-kur2001} have extended it 
for the case that the Frobenius normal form has multiple companion blocks and for computing generalized eigenvectors. 
This approach is considered efficient when the characteristic 
polynomial of the matrix has only one or a small number of irreducible factors. 
However, when the characteristic polynomial has many irreducible factors while one may wish to compute only the eigenspace corresponding to a specific factor, the need to compute the Frobenius normal form of the entire matrix can be a limitation.

We introduce a concept of \emph{Jordan-Krylov basis} and a new method, 
called \emph{Jordan-Krylov elimination} in this paper, for computing a Jordan-Krylov basis. 
We show that the use of minimal annihilating polynomials allows us to design an effective method 
for computing generalized eigenspaces. 
Notably, the resulting algorithm does not involve linear system solving or normal form computation.
Our method has the following features.
First, a basis of the generalized eigenspace is computed as Jordan chains.
Second, each Jordan chain has a unified representation as 
a vector consisting of polynomials with rational coefficients
in a symbol $\lambda$ denoting the associated 
eigenvalue $\alpha$ and its conjugates.
We emphasize that our representation gives a relationship among the 
generalized eigenspaces associated to $\alpha$ and all its conjugates.
Third, in implementing the algorithms, the implementation of arithmetic in an algebraic 
extension field is not necessary. 
In part of our implementation, we perform calculations in the quotient ring of a polynomial ring, which is related to operations in an algebraic extension field. 
However, these calculations can be implemented solely with operations in the polynomial ring over the field of rational numbers. 
All other operations are also performed in the polynomial ring over the field of rational numbers.

We believe that our algorithm could be applied to several methods that require generalized eigenspaces and their structures in the presence of multiple eigenvalues. 
For example, it could be useful when solving 
a zero-dimensional system of polynomial equations 
as the eigenproblem 
(\cite{CLO2005}),
or 
finding formal power series solutions of 
ordinary differential equations (\cite{tur1955}, \cite{bar1999}).

We remark that the essential ideas in this paper, such as the use of 
the minimal annihilating polynomials for computing generalized eigenvectors
were reported in \cite{taj2013}. 
We have revisited the original ideas in depth and sought to optimize the efficiency of the algorithm.
As a result, we have come to the concept of {\JKbasis} and 
the algorithm of {\JKelim} based on it for computing generalized eigenspaces efficiently.

This paper is organized as follows.
In \Cref{sec:minimal-annihilating-polynomial-generalized-eigenvectors},
we give a formula that represents Jordan chains of generalized eigenvectors
as a polynomial with rational coefficients.  
In \Cref{sec:kernel-f(A)l-krylov-generalized-eigenspace}, 
we introduce the concept of {\JKbasis}. 
We show that the computation of Jordan chains can be reduced to that of a
{\JKbasis}. 
In \Cref{sec:jordan-krylov-elimination}, 
we give a method called {\JKelim}, 
and we present the resulting algorithm that computes generalized eigenspaces
from minimal annihilating polynomials.
In \Cref{sec:time-complexity}, we discuss the time complexity of the algorithm.
In \Cref{sec:examples}, 
we give an example.
Finally, in \Cref{sec:exp}, 
we give the result of a benchmark test.

\section{Generalized eigenvectors and Jordan chains}
\label{sec:minimal-annihilating-polynomial-generalized-eigenvectors}

Let us begin by recalling a result given in
\cite{taj-oha-ter2018b}
on algebraic properties of eigenvectors.
Let $\K\subset\C$ be a computational subfield, and $A$ a square matrix of order $n$
over $\K$.
Consider 
a monic irreducible factor $f(\lambda)$ of degree $d$ in $\K[\lambda]$ of the characteristic polynomial $\chi_A(\lambda)$ of $A$, with $m$ the multiplicity of $f(\lambda)$ in 
$\chi_A(\lambda)$.
Let $\pi_A(\lambda)$ be the minimal polynomial of $A$, $\lbar$ the multiplicity of $f(\lambda)$ in $\pi_A(\lambda)$.
Let $\alpha_1,\alpha_2,\ldots,\alpha_d$ be the roots of $f(\lambda)$ in 
$\C$.

Consider a bivariate polynomial
$\psi_f(\mu,\lambda)$ associated to $f(\lambda)$, defined as 
\begin{equation}
  \label{eq:symmetric-polynomial}
  \psi_f(\mu,\lambda)=\frac{f(\mu)-f(\lambda)}{\mu-\lambda} \in K[\mu,\lambda].
\end{equation}

Since the rank of the matrix $f(A)$ is less than $n$, there exists a nonzero vector 
$\w\in\K^n$
satisfying $f(A)\w=\0$. 
For that $\w$, 
we define 
\[
  \p(\lambda,\w)=\psi_f(A,\lambda E)\w,
\]
where $E$ is the identity matrix of order $n$.
Then, $\p(\lambda,\w)$ satisfies the following lemma (\cite{taj-oha-ter2018b}).

\begin{lemma}
  \label{lem:eigenvector}
  Let $\w\in\K^n$ be a nonzero vector such that $f(A)\w=\0$. 
  Then, for $i=1,2,\ldots,d$,
  \begin{enumerate}
    \item $\p(\alpha_i,\w)\ne\0$.
    \item $(A-\alpha_iE)\p(\alpha_i,\w)=\0$.
    \item $\p(\alpha_1,\w),\dots,\p(\alpha_d,\w)$ are linearly independent over $\C$.
  \end{enumerate}
\end{lemma}
\begin{proof}
  By the definition of $\psi_f(\mu,\lambda)$, we have
  \[
    (A-\alpha_iE)\p(\alpha_i,\w)=(A-\alpha_iE)\psi_f(A,\alpha_i E)\w=f(A)\w =\0.
  \]
  For an eigenvalue $\alpha_i\in\C$, consider a vector
  \[
    \q(\alpha_i,\w) = \frac{1}{\psi_f(\alpha_i,\alpha_i)} \p(\alpha_i,\w).
  \]
  Note that, since 
  \[
    \psi_f(\mu,\alpha_i)=(\mu-\alpha_1)\cdots(\mu-\alpha_{i-1})(\mu-\alpha_{i+1})\cdots(\mu-\alpha_d),
  \]
  we have $\psi_f(\alpha_i,\alpha_i)\ne 0$.
  With Lagrange's interpolation formula, we have
  \[
    \sum_{i=1}^{d} \frac{\psi_f(\mu,\alpha_i)}{\psi_f(\alpha_i,\alpha_i)}=1,
  \]
  which yields
  \[
    \sum_{i=1}^{d} \q(\alpha_i,\w)=\w.
  \]
  As a result, we have 
  \[
    \alpha_1^k \q(\alpha_1,\w) + \alpha_2^k \q(\alpha_2,\w) + 
    \cdots + \alpha_d^k \q(\alpha_d,\w) = A^k \w,
    \quad k=0,1,2,\ldots,d-1,
  \]
  which is rewritten as
  \[
    \left[
      \begin{array}{llll}
        1 & 1 & \cdots & 1 \\
        \alpha_1 & \alpha_2 & \cdots & \alpha_d \\
        \vdots  & \vdots &  & \vdots \\
        \alpha_1^{d-1} & \alpha_2^{d-1} & \cdots & \alpha_d^{d-1}
      \end{array}
    \right]
    \begin{bmatrix}
      \q(\alpha_1, \w) \\ \q(\alpha_2, \w) \\ \vdots \\ \q(\alpha_d, \w)
    \end{bmatrix}
    =
    \begin{bmatrix}
      \w \\ A \w \\ \vdots \\ A^{d-1} \w
    \end{bmatrix}
    .
  \]
  (For simplicity, we have used the Vandermonde matrix to show linear independence.)
   In the above formula, the determinant of the matrix on the left-hand side is not 
  equal to zero. 
  Furthermore, since $f(\lambda)$ is irreducible over $\K$,
  the vectors 
  $\w,A\w,\ldots,A^{d-1}\w$ are linearly independent over $\K$,
  so are
  $\q(\alpha_1,\w),\q(\alpha_2,\w),\ldots,\q(\alpha_d,\w)$ 
  over $\C$.
  Accordingly, the vectors \linebreak
  $\p(\alpha_1,\w),\p(\alpha_2,\w),\ldots,\p(\alpha_d, \w)$ are also linearly 
  independent over $\C$ and $\p(\alpha_i,\w) \ne \0$, 
  which completes the proof.
\end{proof}

Notice that the vector $\p(\lambda,\w)$ representing eigenvectors associated to the eigenvalues 
$\alpha_1,\alpha_2,\dots,\alpha_d$ is a polynomial in $\lambda$ of degree at most $d-1$. In this sense, 
this representation has no redundancy.
In this section, we generalize the above results and give a formula that represents a Jordan chain of generalized eigenvectors.

For $1\le\ell\le\bar{\ell}$, let
\[
  \ker f(A)^{\ell} = \{\u \in \K^n \mid f(A)^\ell \u = \0\},
\]
with $\ker f(A)^0=\{\0\}$.
There exists 
an ascending chain 
\begin{equation*}
  \label{eq:kerf(A)-ascending-chain}
  \{\0\} \subset \ker f(A) \subset \ker f(A)^2 \subset \cdots \subset \ker f(A)^{\ell} \subset \cdots  \subset \ker f(A)^{\bar{\ell}},
\end{equation*}
of vector spaces over $\K$.

\begin{definition}
  For $\0\ne\u\in\ker f(A)^{\bar{\ell}}$ and $1\le\ell\le\bar{\ell}$, 
  the \emph{rank} of $\u$ with respect to $f(\lambda)$,
  denoted by $\rank_f\u$, is equal to $\ell$ if 
  $\u\in\ker f(A)^{\ell}\setminus\ker f(A)^{\ell-1}$. 
\end{definition}

\begin{definition}
  For $1\le k\le\ellbar$, 
  let $\psi_f^{(k)}(\mu,\lambda)=(\psi_f(\mu,\lambda))^k \mod f(\lambda)$,
  in which $(\psi_f(\mu,\lambda))^k$ is regarded as an element in $\K[\lambda][\mu]$ and 
  the coefficients in $\psi_f^{(k)}(\mu,\lambda)$ are defined as the remainder of 
  those in $(\psi_f(\mu,\lambda))^k$ divided by $f(\lambda)$.
\end{definition}

Let $\u\in\ker f(A)^{\lbar}$ be a nonzero vector of $\rank_f\u=\ell$.
For $k=1,\dots,\ell$, we define
\begin{equation}
  \label{pklambda}
  \p^{(k)}(\lambda,\u) = \psi_f^{(k)}(A,\lambda E)f(A)^{\ell-k}\u.    
\end{equation}
Then, we have the following lemmas.

\begin{lemma}
  \label{lem:p(k)-p(k-1)}
  Assume that $\rank_f\u=\ell$. Then, for $i=1,2,\ldots,d$ and $1\le k\le \ell$, 
  \[
    (A-\alpha_i E)\p^{(k)}(\alpha_i,\u) = \p^{(k-1)}(\alpha_i,\u)
  \]
  holds.
\end{lemma}
\begin{proof}
  Since
  \[
    (A-\alpha_iE)\psi^{(k)}_f(A,\alpha_i E) = \psi^{(k-1)}_f(A,\alpha_i E)f(A),
  \]
  we have
  \[
    (A-\alpha_iE)\psi^{(k)}_f(A,\alpha_i E)f(A)^{\ell-k} = \psi^{(k-1)}_f(A,\alpha_i E)f(A)^{\ell-(k-1)}.
  \]
  By multiplying $\u$ on both sides from the right, we have
  \[
    (A-\alpha_i E)\p^{(k)}(\alpha_i,\u) = \p^{(k-1)}(\alpha_i,\u),
  \]
  which proves the lemma.
\end{proof}

\begin{lemma}
  \label{lem:generalized-eigenvectors}
  Assume that $\rank_f\u=\ell$. Then, for $i=1,2,\ldots,d$ and $1\le k\le \ell$, 
  \[
    (A-\alpha_i E)^{k-1} \p^{(k)}(\alpha_i,\u) \ne \0 \qquad \text{and} \qquad (A-\alpha_i E)^k \p^{(k)}(\alpha_i,\u) = \0 
  \]
  hold.
\end{lemma}
\begin{proof}
  By Lemma~\ref{lem:p(k)-p(k-1)}, we have
  \begin{align*}
    (A-\alpha_i E)^{k-1}\p^{(k)}(\alpha_i,\u) &= \p^{(1)}(\alpha_i,\u),\\
    (A-\alpha_i E)^{k}\p^{(k)}(\alpha_i,\u) &= (A-\alpha_i E)\p^{(1)}(\alpha_i,\u),
  \end{align*}
  where $\p^{(1)}(\alpha_i,\u) = \p(\alpha_i,f(A)^{\ell-1}\u)$.
  Let $\w=f(A)^{\ell-1}\u$. Then
  $\w\ne \0$ and $f(A)\w = \0$ since $\rank_f\u=\ell$.
  Thus, by Lemma~\ref{lem:eigenvector}, we have
  $\p^{(1)}(\alpha_i,\u) \ne \0$ and $(A-\alpha_i E)\p^{(1)}(\alpha_i,\u) = \0$.
  This completes the proof.
\end{proof}

Lemma~\ref{lem:generalized-eigenvectors} says that 
$\p^{(k)}(\lambda,\u)$ gives a representation of generalized eigenvectors of $A$ of rank $k$.
Summarizing the above, we have the following theorem.

\begin{theorem}
  Assume that $\rank_f\u=\ell$. Then, for $i=1,2,\ldots,d$,
  \begin{equation}
    \label{eq:jordan-chain}
    \{
      \p^{(\ell)}(\alpha_i,\u), \p^{(\ell-1)}(\alpha_i,\u), \ldots, \p^{(1)}(\alpha_i,\u)
    \}
  \end{equation}
  gives a Jordan chain of length $\ell$.
\end{theorem}

Accordingly,
$\{
  \p^{(\ell)}(\lambda,\u), \p^{(\ell-1)}(\lambda,\u), \ldots, \p^{(1)}(\lambda,\u)
\}$
can be regarded as a representation of 
the Jordan chain \eqref{eq:jordan-chain}.
That is, from a vector $\u$ of rank $\ell$ in $\ker f(A)^{\ell}\subset\K^n$, 
the representation of the Jordan chain is derived
in terms of $\p^{(k)}(\lambda,\u)$.

For the discussion in the following example, we introduce the concept of the \emph{minimal annihilating polynomial}. The key role of the minimal annihilating polynomial in the computation of eigenvectors will be discussed in \Cref{sec:jordan-krylov-elimination} in detail.

\begin{definition}[The minimal annihilating polynomial]
  For $\u\in\K^n$, let $\pi_{A,\u}(\lambda)$ be the monic generator of the principal 
  ideal $\ann_{K[\lambda]}(\u)=\{h(\lambda)\in K[\lambda] \mid h(A)\u=\bm{0}\}$.
  Then, $\pi_{A,\u}(\lambda)$ is called \emph{the minimal annihilating polynomial of $\u$
  with respect to $A$}.
\end{definition}

We note that, in the following example, standard unit vectors $\e_1,\dots,\e_n$ are used.
Let us denote $\pi_{A,j}(\lambda)=\pi_{A,\e_j}(\lambda)$ for 
the standard unit vector $\e_j$ and call it the $j$-th unit minimal annihilating polynomial. 

\begin{expl}
  \label{expl:eigenvector-1}
  Let $f(\lambda)=\lambda^2+\lambda+5$ and
  \[
    A=
    \begin{pmatrix}
      0 & 0 & 0 & 0 & 0 & -125\\
      1 & 0 & 0 & 0 & 0 & -75\\
      0 & 1 & 0 & 0 & 0 & -90\\
      0 & 0 & 1 & 0 & 0 & -31\\ 
      0 & 0 & 0 & 1 & 0 & -18\\ 
      0 & 0 & 0 & 0 & 1 & -3
    \end{pmatrix}
    ,
  \]
  the companion matrix of $f(\lambda)^3$.
  The characteristic polynomial $\chi_A(\lambda)$ and the unit minimal 
  annihilating polynomials 
  $\pi_{A,j}(\lambda)$ ($j=1,2,\dots,6$) are 
  all equal to $f(\lambda)^3$.
  We see that $\e_1\in\ker f(A)^3$ and $\rank_f\e_1=3$.
Now, let us compute the generalized eigenvector $\p^{(k)}(\lambda,\e_1)$
associated to $f(\lambda)$. For
\begin{equation}
  \label{eq:psi-f}
  \psi_f(\mu,\lambda)=\mu+\lambda+1,
\end{equation}
$\psi_f(\mu,\lambda)^2$ and $\psi_f(\mu,\lambda)^3$ are polynomials of 
degree 2 and 3 with respect to $\lambda$, respectively. 
Let $\psi_f^{(2)}(\mu,\lambda)$ and $\psi_f^{(3)}(\mu,\lambda)$ be the remainders of
$\psi_f(\mu,\lambda)^2$ and $\psi_f(\mu,\lambda)^3$
upon division by $f(\lambda)$ with respect to $\lambda$, calculated as
\begin{equation}
  \label{eq:psi-2}
  \begin{split}
    \psi_f^{(2)}(\mu,\lambda) &= \mu^2+(2\lambda+2)\mu+\lambda-4,\\
    \psi_f^{(3)}(\mu,\lambda) &= \mu^3 +(3\lambda+3)\mu^2+(3\lambda-12)\mu-4\lambda-9,      
  \end{split}
\end{equation}
respectively. Then, the generalized eigenvectors are computed 
with $\e_1$
as
\begin{align*}
  \p^{(3)}(\lambda,\e_1) &= 
  \psi_f^{(3)}(A,\lambda E)\bm{e}_1
  = {}^t(-4,3,3,0,0,0)\lambda +{}^t(-9,-12,3,1,0,0),\\
  \p^{(2)}(\lambda,\e_1) &= 
  \psi_f^{(2)}(A,\lambda E)f(A)\bm{e}_1
  = {}^t(5,11,3,2,0,0)\lambda +{}^t(-20,6,3,3,1,0),\\
  \p^{(1)}(\lambda,\e_1) &= 
  \psi_f^{(1)}(A,\lambda E)f(A)^2\bm{e}_1
  = {}^t(25,10,11,2,1,0)\lambda +{}^t(25,35,21,13,3,1),  
\end{align*}
which satisfy
\begin{alignat*}{2}
  (A-\lambda E)^2\p^{(3)}(\lambda,\e_1) &\ne \bm{0},\quad &
  (A-\lambda E)^3\p^{(3)}(\lambda,\e_1) &= \bm{0}, 
  \\
  (A-\lambda E)\p^{(2)}(\lambda,\e_1) &\ne \bm{0},&
  (A-\lambda E)^2\p^{(2)}(\lambda,\e_1) &= \bm{0}, 
  \\
  \p^{(1)}(\lambda,\e_1) &\ne \bm{0}, &
  (A-\lambda E)\p^{(1)}(\lambda,\e_1) &= \bm{0}, 
\end{alignat*}
respectively.
Thus, the set
$\{\p^{(3)}(\lambda,\e_1),
\p^{(2)}(\lambda,\e_1),
\p^{(1)}(\lambda,\e_1)
\}$ 
is a Jordan chain of length 3.
\end{expl}

\section{{\JKbasis} in $\ker f(A)^{\lbar}$}
\label{sec:kernel-f(A)l-krylov-generalized-eigenspace}

The discussion in the previous section shows that 
a vector $\u\in\ker f(A)^{\lbar}\subset \K^n$ of rank $\ell$
gives rise to a representation of the corresponding Jordan chains of
length $\ell$.
In this section, we show that there exists a finite subset 
$\cB$ of $\ker f(A)^{\lbar}$ such that 
any generalized eigenvector can be represented
as a linear combination of vectors 
$\p^{(k)}(\lambda,\v)$ with $\v\in\cB$.

For $\w\in\K^n$, a vector space
\[
  L_A(\w) = \Span_{\K}\, \{ A^k \w \mid k=0,1,2,\ldots \}
\]
is called the \emph{Krylov vector space}.
For $\u\in\ker f(A)^{\ell}$, 
since $A$ and $f(A)$ commute, $\rank_f\u=\ell$ implies that 
$L_A(\u)\subset \ker f(A)^{\ell}$.
For $d=\deg f(\lambda)$, let
\[
  \cL_{A,d}(\u) = \{\u,A\u,A^2\u,\dots,A^{d-1}\u\}
\]
and
\[
L_{A,d}(\u) = \Span_{\K}\cL_{A,d}(\u).
\]
Since $f(\lambda)$ is irreducible, 
$\cL_{A,d}(\u)$
is linearly independent, and
the rank of the non-trivial elements in $L_{A,d}(\u)$ is equal to $\ell$.
Let $\cL_A(\u)=\cL_{A,d}(\u)\cup\cL_{A,d}(f(A)\u)\cup\cdots\cup\cL_{A,d}(f(A)^{\ell-1}\u)$.
Clearly, $\cL_A(\u)$ is also linearly independent.
Then, it holds that
\[
L_{A}(\u) = L_{A,d}(\u) \oplus f(A) L_{A,d}(\u) \oplus \cdots \oplus f(A)^{\ell-1} L_{A,d}(\u),
\]
thus $\dim_K L_{A}(\u)=\ell d$.
(In a more straightforward manner, we see that $\dim_K L_{A}(\u)=\ell d$ by the fact that $f(\lambda)^{\ell}$ is the minimal annihilating polynomial of $\u$.)
Note that, for $1\le \ell'\le \ell-1$, the rank of the non-trivial elements 
in $f(A)^{\ell'}L_{A,d}(\u)$ is equal to $\ell-\ell'$.

Consider the subspace spanned by a Jordan chain
\[
  P_A(\alpha_i,\u) = \Span_{\C}\, \{ 
    \p^{(\ell)}(\alpha_i,\u), \p^{(\ell-1)}(\alpha_i,\u), \ldots, \p^{(1)}(\alpha_i,\u)
    \},
\]
where $\alpha_1,\alpha_2,\dots,\alpha_d$ are the roots of $f(\lambda)$ in $\C$.
Let
\[
P_A(\u) = P_A(\alpha_1,\u) \oplus P_A(\alpha_2,\u) \oplus \cdots \oplus P_A(\alpha_d,\u).
\]
We have the following property on $P_A(\u)$.

\begin{lemma}
  \label{lem:p_a(u)}
  $P_A(\u) = \C \otimes_{\K} L_{A}(\u)$ holds.  
\end{lemma}
\begin{proof}
  Note that, for $1\le k\le\ell$, $\p^{(k)}(\alpha_i,\u)\in \C\otimes_{\K} L_A(\u)$ 
  since $\p^{(k)}(\alpha_i,\u)$ is expressed as $h(A)\u$
	for a certain polynomial $h$ of degree less than $d\ell$. 
  We also mention that $L_A(\u)$ is invariant under $h(A)$ for any polynomial $h(\lambda)\in\K[\lambda]$. 
  By the definition of $P_A(\alpha_i,\u)$, 
  it implies that 
  $P_A(\u) \subset \C \otimes_{\K} L_{A}(\u)$,
  and $\dim_{\C} P_A(\alpha_i,\u)=\ell$ implies that $\dim_{\C} P_A(\u)=\ell d$.
  On the other hand, from
  $\dim_K L_A(\u)=\ell d$ and $L_A(\u)\subset {\K}^n$, we have
  $\dim_{\C} \C \otimes_{\K} L_A(\u)=\ell d$. 
  Since the dimensions of $P_A(\u)$ and $\C \otimes_{\K} L_{A}(\u)$ are equal,
  we have $P_A(\u) = \C \otimes_{\K} L_{A}(\u)$.
  This completes the proof.
\end{proof}

Notice that the decomposition 
$P_A(\alpha_1,\u)\oplus P_A(\alpha_2,\u)\oplus\cdots\oplus P_A(\alpha_d,\u)$
of $\C\otimes_{\K}L_{A}(\u)$ can be regarded as a spectral decomposition of $P_A(\u)$.

Let us recall that a sum $V_1+V_2+\cdots+V_r$ of vector spaces $V_1,V_2,\dots,V_r$ 
satisfying the \dsc,
\[
  (V_1+\cdots+V_{i-1}+V_{i+1}+\cdots+V_r)\cap V_i = \{\0\},\quad
  i=1,2,\dots,r,
\]
is written as
$V_1\oplus V_2 \oplus \cdots \oplus V_r$.

Now, we introduce notions of {\KrylovGS}, {\JKindepce} and
{\JKbasis} as follows.

\begin{definition}[{\KrylovGS}]
  Let $W$ be a subspace of $\K^n$.
  A subset $\cW$ of $W$
  is a \emph{\KrylovGS} 
  with respect to $A$
  if the sum $W=\sum_{\w\in\cW}L_A(\w)$ holds.
  In particular, the empty set is a {\KrylovGS} of the trivial vector space $\{\0\}$.
\end{definition}

\begin{definition}[{\JKindepce}]
  A finite set $\cW$ of vectors is \emph{\JKindep} 
  with respect to $A$ if 
  Krylov vector spaces $\{L_A(\w)\mid \w\in\cW\}$ satisfy the {\dsc}.
  In particular, the empty set is regarded as {\JKindep}.
\end{definition}

\begin{definition}[{\JKbasis}]
  A finite {\KrylovGS} $\cW$ of $W$
  is a \emph{\JKbasis} 
  if it is {\JKindep}.
\end{definition}

\begin{theorem}
  \label{thm:jkbasis}
  There exists a {\JKbasis}
  $\cB$ of $\ker f(A)^{\bar{\ell}}$ with respect to $A$.
  Furthermore, the number of elements in $\cB$ of each rank is uniquely determined.
\end{theorem}

A {\JKbasis} is a key concept in our approach.
In \Cref{sec:jkbasis-proof},
we give a proof of \Cref{thm:jkbasis} in a constructive manner 
which plays a crucial role
in our algorithm for computing a {\JKbasis}.

For convenience, we introduce a superscript $(\ell)$ for a finite subset of nonzero vectors
$\cA\subset\ker f(A)^{\lbar}$,
as $\cA^{(\ell)}=\{\u\in\cA\mid\rank_f\u=\ell\}$.
Then
$
  \cA=\cA^{(1)}\cup\cdots\cup\cA^{(\lbar)}.
$

\Cref{thm:jkbasis} and Lemma~\ref{lem:p_a(u)} immediately imply the following lemma.

\begin{lemma}
  Assume that $\{\b_1,\b_2,\ldots,\b_r\}\subset\ker f(A)^{\lbar}$ is {\JKindep}.
  Then, the following holds.
  \begin{enumerate}
    \item $P_A(\alpha_i,\b_1),P_A(\alpha_i,\b_2),\ldots,P_A(\alpha_i,\b_r)$ satisfy the {\dsc} for $i=1,2,\ldots,d$. 
    \item $P_A(\b_1),P_A(\b_2),\ldots,P_A(\b_r)$ satisfy the {\dsc}.
  \end{enumerate}
\end{lemma}

As a result, we have the following theorem.

\begin{theorem}
  Let 
  $\cB=\cB^{(\bar{\ell})}\cup\cB^{(\bar{\ell}-1)}\cup\cdots\cup\cB^{(1)}$
  be a {\JKbasis} of $\ker f(A)^{\lbar}$.
  Then, for $i=1,2,\ldots,d$, the following hold.
  \begin{enumerate}
    \item For $\ell=1,\dots,\lbar$, a direct sum $\displaystyle \bigoplus_{\b\in \cB^{(\ell)}}P_A(\alpha_i,\b)$ 
    is spanned by Jordan chains of length $\ell$.
    \item A direct sum $\displaystyle \bigoplus_{\b\in \cB}P_A(\alpha_i,\b)$
    gives a generalized eigenspace of $A$ associated to the eigenvalue $\alpha_i$.
    \item %
    $\displaystyle \bigoplus_{\b\in \cB}P_A(\b)\simeq \C \otimes_{\K} \ker f(A)^{\bar{\ell}}$.
  \end{enumerate}
\end{theorem}

\subsection{A proof of \Cref{thm:jkbasis}}
\label{sec:jkbasis-proof}

We first give a lemma, which will play a key role in our approach.

\begin{lemma}
  \label{lem:linsep}
  Assume that $\{\b_1,\b_2,\ldots,\b_r\}\subset\ker f(A)^{\lbar}$ is {\JKindep} and 
  $\v\in \ker f(A)^{\bar{\ell}}$ satisfies
  \[
  f(A)^{\rank_f \v-1}\v \not\in \bigoplus_{i=1}^r L_A(f(A)^{\rank_f \b_i-1}\b_i).
  \]
  Then, $L_A(\b_1), \ldots, L_A(\b_r), L_A(\v)$ satisfy the {\dsc}.
\end{lemma}
\begin{proof}
  Suppose, to the contrary, that $L_A(\b_1),\ldots,L_A(\b_r), L_A(\v)$ do not satisfy the {\dsc}.
  Let $\b_i'=f(A)^{\rank_f \b_i-1}\b_i$ and $\v'=f(A)^{\rank_f \v-1}\v$.
  We have
  \[
  L_A(\v) \cap \bigoplus_{i=1}^r L_A(\b_i)\ne\{\0\},
  \]
  thus there exists a nonzero vector $\u\in L_A(\v)\cap\bigoplus_{i=1}^r L_A(\b_i)$.
  Then, $\u'=f(A)^{\rank_f \u -1}\u$ satisfies $\rank_f\u'=1$, hence 
  $\u'\in L_A(\v')\cap\bigoplus_{i=1}^r L_A(\b_i')$. 
  By $\u'\in L_A(\v')$, we have $L_A(\u')\subset L_A(\v')$.
  Since $f(\lambda)$ is irreducible, 
  $\dim_K L_A(\w)=d$ for any $\w\in\ker f(A)\setminus\{0\}$,
  which implies that $L_A(\u')=L_A(\v')$ hence $\v'\in L_A(\u')$.
  Now, $\u'\in \bigoplus_{i=1}^r L_A(\b_i')$ implies that 
  $L_A(\u')\subset\bigoplus_{i=1}^r L_A(\b_i')$. Thus, we have
  $\v'\in\bigoplus_{i=1}^r L_A(\b_i')$, which contradicts the assumption.
  This completes the proof.
\end{proof}

For simplicity of notation, 
for $\cU\subset\K^n$, let $\cL_A(\cU)=\bigcup_{\u\in\cU}\cL_A(\u)$.
Now we are ready to give a proof of \Cref{thm:jkbasis}.
Let $\cG$ be a finite generating set of $\ker f(A)^{\lbar}$; then 
clearly $\lbar= \max \{ \rank_f \u \mid \u \in \cG \}$.
In the following, for $\ell=\lbar,\ldots,2,1$ in descending order,
let us construct a {\JKbasis} recursively.
For $\ell$ with $1\le\ell\le\lbar$, 
assume that we have constructed a finite generating set 
$\cL_A(\cF_{\ell})\cup\cG_{\ell}$ of $\ker f(A)^{\lbar}$
satisfying \\
Condition A:
\begin{enumerate}
  \item $\cF_{\ell}$ is a {\JKindep} set such that any 
  $\v\in\cF_{\ell}$ satisfies $\rank_f\v>\ell$.
  \item Any $\u\in\cG_{\ell}$ satisfies $\rank_f\u \le \ell$.
\end{enumerate}
For $\ell=\lbar$, put $\cF_{\lbar}=\emptyset$ and $\cG_{\lbar}=\cG$. 
Then, we see that 
$\cF_{\lbar}$ and $\cG_{\lbar}$ satisfy Condition A.

Now, we construct a finite generating set 
$\cL_A(\cF_{\ell-1})\cup\cG_{\ell-1}$ of $\ker f(A)^{\lbar}$ 
as follows.
Put
$\{\v_1,\v_2,\ldots,\v_{k_\ell}\}=\{\u\in\cG_{\ell}\mid\rank_f\u=\ell\}$, 
and let 
$\cD_{\ell,1}=\cF_{\ell}$ and 
$\cH_{\ell,1}=\cG_{\ell}\setminus\{\v_1,\v_2,\ldots,\v_{k_\ell}\}$. 
For $i=1,2,\ldots,k_{\ell}$, according to the relationship between 
a subspace
\begin{equation}
  \label{eq:w_i}
  W_{\ell,i} = \bigoplus_{\b\in\cD_{\ell,i}} L_A(f(A)^{\rank_f\b-1}\b)  
\end{equation}
and the vector $f(A)^{\ell-1}\v_i$, define $\cD_{\ell,i+1}$ and 
$\cH_{\ell,i+1}$, divided into the following cases:
\begin{description}
  \item[Case 1.] In the case $f(A)^{\ell-1}\v_i\not\in W_{\ell,i}$: Lemma~\ref{lem:linsep}
  tells us that $\{ L_A(\b) \mid \b\in \cD_{\ell,i}\} \cup \{ L_A(\v_i)\}$ satisfy the {\dsc}.
  Thus, let $\cD_{\ell,i+1}=\cD_{\ell,i}\cup \{\v_i\}$ and
  $\cH_{\ell,i+1}=\cH_{\ell,i}$.
  \item[Case 2.] In the case $f(A)^{\ell-1}\v_i\in W_{\ell,i}$: the membership can be expressed as
  \begin{equation}
    \label{eq:f(A)^(l-1)vi-linear-dependence}
    \0 = f(A)^{\ell-1}\v_i - \sum_{\b\in\cD_{\ell,i}} \sum_{j=0}^{d-1} c_{\b,j} A^j f(A)^{\rank_f \b-1}\b.
  \end{equation}
  Define
  \begin{equation}
    \label{eq:vi-linear-dependence}
    \r = \v_i - \sum_{\b\in\cD_{\ell,i}} \sum_{j=0}^{d-1} c_{\b,j} A^j f(A)^{\rank_f\b-\ell}\b.
  \end{equation}
  Note that, by the construction of $\cD_{\ell,i}$, any vector $\b\in\cD_{\ell,i}$ satisfies $\rank_f\b\ge\ell$.
  Then, do one of the following:
  \begin{description}
    \item[Case 2-a.] If $\r\ne \0$, then $\rank_f \r < \ell$; let
    $\cD_{\ell,i+1}=\cD_{\ell,i}$ and $\cH_{\ell,i+1}=\cH_{\ell,i}\cup\{\r\}$.
    \item[Case 2-b.] If $\r=\0$, then $\v_i\in \bigoplus_{\b\in\cD_{\ell,i}} L_A(\b)$
    and $\v_i$ is unnecessary; discard it and 
    let $\cD_{\ell,i+1}=\cD_{\ell,i}$ and $\cH_{\ell,i+1}=\cH_{\ell,i}$.
  \end{description}
\end{description}
Note that, in any case above, 
\[
  \cL_A(\cD_{\ell,i+1}) \cup (\{\v_{i+1},\ldots,\v_{k_\ell}\}\cup\cH_{\ell,i+1})
\]
is a generating set satisfying Condition A. 
Then, after deriving a sequence of tuples of sets 
$(\cD_{\ell,1},\cH_{\ell,1}),(\cD_{\ell,2},\cH_{\ell,2}),\ldots,
(\cD_{\ell,k_{\ell}+1},\cH_{\ell,k_{\ell}+1})$,
by letting 
$\cF_{\ell-1}=\cD_{\ell,k_{\ell}+1}$, $\cG_{\ell-1}=\cH_{\ell,k_{\ell}+1}$,
we see that $\cF_{\ell-1}$ and $\cG_{\ell-1}$ also satisfy Condition A.

After the above computation for $\ell=\bar{\ell},\ldots,2,1$, a generating set 
$\cL_A(\cF_{0})\cup\cG_{0}=\cL_A(\cF_{0})$ of 
$\ker f(A)^{\lbar}$ is computed (note that $\cG_{0}=\emptyset$), 
which shows that $\cF_{0}$ is a {\JKbasis} $\cB$. 
This completes the proof.

\begin{remark}
  \label{rem:m=lbar}
  Note that the multiplicity $m$ of $f(\lambda)$ in the characteristic polynomial of $A$
   is equal to the sum of the lengths of the linearly independent Jordan chains.
  In the case $m=\lbar$, 
  let $\v_1 \in \{ \u \in \cG \mid \rank_f \u = \lbar\}$ and
  $\cF_{\lbar}=\{\v_1\}$.
  Then, we have
  \begin{equation}
    \label{eq:m=lbar}
    L_A(\v_1) = \ker f(A)^{\bar{\ell}},
  \end{equation}
  thus $\cF_{\lbar} = \{\v_1\}$ gives a {\JKbasis} of $\ker f(A)^{\bar{\ell}}$.
\end{remark}

\section{Jordan-Krylov elimination}
\label{sec:jordan-krylov-elimination}

In this section, based on the concept of {\JKbasis}, we provide an algorithm for 
computing generalized eigenvectors.

Let $f(\lambda)$ be a monic
irreducible factor of the minimal polynomial $\pi_A(\lambda)$ of $A$
and let $\lbar$ be its multiplicity.
We assume that a {\KrylovGS} of $\ker f(A)^{\lbar}$ is already given.
In \Cref{sec:jkbasis-computation}, we look at the proof of \Cref{thm:jkbasis} given in 
the previous section from the point of view of symbolic computation, and 
we introduce a method called {\JKelim} for computing a {\JKbasis}.
In \Cref{sec:jkelim-refined}, we refine the method and give an algorithm that computes a
{\JKbasis} from the {\KrylovGS}. 
In \Cref{sec:generalized-eigenspace}, we present a resulting algorithm for computing
the generalized eigenspace associated to the eigenvalues of an irreducible factor
of the characteristic polynomial $\chi_A(\lambda)$ of $A$.

Although we assumed that a Krylov generating set is already available, in practice a method for 
obtaining a Krylov generating set is also necessary.
In a series of our previous works, we have made use of the minimal annihilating polynomials.
In \Cref{sec:ker-f(A)lbar-generators}, we introduce a method for computing a Krylov generating set 
using the minimal annihilating polynomials.

\subsection{Computing a {\JKbasis} of $\ker f(A)^{\lbar}$ using {\JKelim}}
\label{sec:jkbasis-computation}

We assume that a {\KrylovGS} 
of $\ker f(A)^{\lbar}$ is given as 
$\cV=\bigcup \cV^{(\ell)}$,
where $\cV^{(\ell)}=\{\v\in\cV\mid \rank_f\v=\ell\}$.
Now, we need to compute 
a {\JKbasis} $\cB$ of $\ker f(A)^{\lbar}$ from $\cV$.

In the proof of \Cref{thm:jkbasis}, we examine the membership 
of a vector $\v'=f(A)^{\rank_f\v-1}\v$ in the vector space $W_{\ell,i}$
along with calculating the coefficients of 
$c_{\b,j}$ of linear combinations. 
It is well known that 
a vector membership problem is solved by using
column reduction of the matrix.
For a finite set of column vectors $\cW=\{\w_1,\w_2,\ldots,\w_k\}$,
define a matrix whose columns consist of the vectors in $\cW$
as $[\, \cW \,] = [\w_1\  \w_2 \ \cdots \ \w_k]$.
Similarly, for a set of vectors $\cW$ and a vector $\v'$, 
define an augmented matrix 
$[\, \cW \mid \v'] = [\w_1\  \w_2 \ \cdots \ \w_k \mid \v']$, in which 
$\v'$ must be placed in the rightmost column.
Then, $\v'\in\Span_K\cW$ if and only if 
there exists a column reduction $[\,\cW \mid \v'] \longrightarrow [\,\cW \mid \0]$.

\begin{remark}
  Let us look at the proof of \Cref{thm:jkbasis}.
  For a vector $\v$ of rank $\ell$,
  we have solved a membership problem
  of the vector $f(A)^{\ell-1}\v$
  in $\Span_{\K}\cL_A(\{f(A)^{\rank_f \b-1}\b\mid \b \in \cD_{\ell,i}\})$ for 
  $\cD_{\ell,i}$.
  For $\b\in\cD_{\ell,i}$, we consider a vector
  $\tilde{\b}=
  \begin{bmatrix}
    f(A)^{\rank_f\b-1}\b \\
    f(A)^{\rank_f\b-\ell}\b
  \end{bmatrix}$
  and define
  $\tilde{D}=
  \begin{bmatrix}
    D' \\ D  
  \end{bmatrix}$
  by $\tilde{D}=[\cL_{\tilde{A},d}(\{\tilde{\b}\mid\b\in\cD_{\ell,i}\})]$,
  where $\tilde{A}$ is a matrix obtained by placing $A$ diagonally as
  $\tilde{A} = \begin{pmatrix} A & 0 \\ 0 & A \end{pmatrix}$.
  Notice that 
  $W_{\ell,i}$ is spanned by the column vectors in $D'$ 
  and the rank of every column vector in $D$ is equal to $\ell$.
  Then, we see that
  $f(A)^{\ell-1}\v\in W_{\ell,i}$
  is equivalent to the existence of a column reduction 
  $[\, \tilde{D} \mid \tilde{\v}] \longrightarrow \begin{bmatrix}D' & \0 \\ D & \r\end{bmatrix}$,
  where $\tilde{\v}=
  \begin{bmatrix}
  \v' \\ \v 
  \end{bmatrix}
  $
  with $\v'=f(A)^{\ell-1}\v$.
\end{remark}

With the above discussions, the proof of \Cref{thm:jkbasis} yields an algorithm
as shown in \Cref{alg:jordan-krylov-elim-0}.

\begin{algorithm}[H]
  \caption{Computing a {\JKbasis} of $\ker f(A)^{\lbar}$}
  \label{alg:jordan-krylov-elim-0}
  \begin{algorithmic}[1]
    \Input{A matrix $A\in\K^{n\times n}$, an irreducible factor $f(\lambda)\in\K[\lambda]$,
    
    a {\KrylovGS} $\cV=\bigcup_{\ell=1}^{\lbar}\cV^{(\ell)}$ of $\ker f(A)^{\lbar}$, 

    the multiplicity $m$ of $f(\lambda)$ in the characteristic polynomial of $A$} 
    \Output{A {\JKbasis} $\cB=\bigcup_{\ell=1}^{\lbar}\cB^{(\ell)}$ of $\ker f(A)^{\lbar}$}
    \State 
      $\displaystyle{
        \tilde{A}\gets
        \begin{bmatrix}
          A & O \\
          O & A            
        \end{bmatrix}
      }$
    \For{$\ell=\lbar,\lbar-1,\dots,1$}
      \State{$\tilde{\cV}^{(\ell)}\gets\{\tilde\v\mid\v\in\cV^{(\ell)}\}$,\quad
      $\cB^{(\ell)}\gets\emptyset$}
    \EndFor
    \For{$\ell=\lbar,\lbar-1,\dots,1$}
      \label{line:jordan-krylov-elim-0:for-begin}
      \label{line:jordan-krylov-elim-0:barcV}
      \If{$\ell=\lbar$}
        \State Choose $\tilde{\v}\in \tilde{\cV}^{(\lbar)}$, \quad
        $\tilde{\cV}^{(\lbar)}\gets\tilde{\cV}^{(\lbar)}\setminus\{\tilde{\v}\}$, \quad
        $\cB^{(\lbar)}\gets\{\v\}$
        \If{$\lbar=m$}
          \State{\Return{$\cB=\cB^{(\lbar)}$}}
          \Comment{Case of \cref{eq:m=lbar}}
        \Else
          \State Do a column reduction $[\cL_{\tilde{A},d}(\tilde{\v})]\longrightarrow \tilde{D}$
          \label{line:jordan-krylov-elim-0:w_0}
        \EndIf
      \EndIf
      \While {$\tilde{\cV}^{(\ell)}\ne\emptyset$} 
        \State Choose $\tilde{\v} \in \tilde{\cV}^{(\ell)}$, \quad
        $\tilde{\cV}^{(\ell)} \gets \tilde{\cV}^{(\ell)} \setminus \{\tilde{\v}\}$
        \State
        Do a column reduction of the rightmost column in the augmented matrix

        \qquad
        $[\tilde{D}\mid\tilde{\v}]\longrightarrow
          \begin{bmatrix}
            D' & \r' \\
            D & \r
          \end{bmatrix}
        $
        \label{line:jordan-krylov-elim-0:elim}
        \If{$\r'\ne\0$} \label{line:jordan-krylov-elim-0:zero-verification} 
          \Comment{Case 1.\ in the proof of \cref{thm:jkbasis}}
          \State $\cB^{(\ell)}\gets\cB^{(\ell)}\cup\{\r\}$
          \State Do a column reduction
           $[\tilde{D}\mid\cL_{\tilde{A},d}(\tilde{\r})]\longrightarrow \tilde{R}$
          \State $\tilde{D}\gets\tilde{R}$
          \label{line:jordan-krylov-elim-0:add-basis}
        \ElsIf{$\r\ne\0$ and $\ell>1$}
          \Comment{Case 2.\ (a) in the proof}
          \State $\ell'\gets\rank_f\r$, 
          $\tilde{\cV}^{(\ell')}\gets\tilde{\cV}^{(\ell')}\cup\{\tilde{\r}\}$
          \label{line:jordan-krylov-elim-0:nextcandidate}
        \EndIf
      \EndWhile
      \If{$\ell>1$}
        \State 
        $\tilde{D}\gets
        \displaystyle{
          \begin{bmatrix}
              E_n & O \\
              O & f(A)
          \end{bmatrix}
          \tilde{D}
        }$
        \Comment{For the next step in the loop}
        \label{line:jordan-krylov-elim-0:C(l)}
      \EndIf
    \EndFor
    \label{line:jordan-krylov-elim-0:for-end}
  \State \Return{$\cB=\bigcup_{\ell=1}^{\lbar}\cB^{(\ell)}$}
  \end{algorithmic}
\end{algorithm}

\subsection{A refined algorithm for computing {\JKbasis}}
\label{sec:jkelim-refined}

In this section, we present \Cref{alg:jordan-krylov-basis} as a refinement of 
the algorithm presented in the last section.

Since the multiplicity $m$ of $f(\lambda)$ in the characteristic polynomial of $A$
is equal to the sum of the lengths of linearly independent Jordan chains, it holds that
\begin{equation}
  \label{eq:rank-condition}
  m=\sum_{\ell=1}^{\lbar}\,\ell\cdot\#\cB^{(\ell)},
\end{equation}
where $\#\cB^{(\ell)}$ denotes the number of elements in $\cB^{(\ell)}$.
If a {\JKindep} set 
satisfies \cref{eq:rank-condition}, this gives 
a {\JKbasis}.
That is, \cref{eq:rank-condition} is a terminating condition 
of the algorithm.

In \Cref{alg:jordan-krylov-elim-0}, 
for a vector $\v$ of rank $\ell$, 
a vector membership problem is reduced to a column reduction of 
an augmented matrix 
using $\tilde{\v}$ defined by
placing vectors $f(A)^{\ell-1}\v$ and $\v$ vertically.
For example, in line~\ref{line:jordan-krylov-elim-0:w_0} 
of \Cref{alg:jordan-krylov-elim-0},
the matrix $[\cL_{\tilde{A},d}(\tilde{\v})]$ is constructed
by multiplying a matrix of size $2n$ from the left.
However, since the matrix is block diagonal, actual multiplication can
be executed by multiplying a matrix of size $n$. 
Thus, in \Cref{alg:jordan-krylov-basis}, 
$[\cL_{\tilde{A},d}(\tilde{\v})]$ is divided into
$[\cL_{A,d}(\v)]$ and
$[\cL_{A,d}(f(A)^{\ell-1}\v)]$.
Then, the column reduction of the matrix
$
\begin{bmatrix}
  W & \r' \\
  S & \r
\end{bmatrix}
=
\begin{bmatrix}
  W & \v' \\
  S & \v
\end{bmatrix}
C
$
is divided into two column reductions
$[W\mid \r']=[W\mid \v']C$ and
$[S\mid \r]=[S\mid \v]C$ with the same matrix $C$,
which is called a \emph{simultaneous column reduction}.
Note that, in the case of $\ell=1$, we have
$W=S$; hence, both parts of the simultaneous column reduction are identical.

\begin{remark}
  \label{remark:termination-condition}
  During the {\JKelim}, there are cases where unnecessary calculations can be eliminated based on the basis that has already been computed, which is included in lines~\ref{line:jordan-krylov-basis:v} and \ref{line:jordan-krylov-basis:ell} of \Cref{alg:jordan-krylov-basis}.
  The former case is described as in \Cref{rem:m=lbar}. 
  In the latter case, if the sum of the lengths of already calculated Jordan chains is equal to the multiplicity $m$ of $f(\lambda)$ in $\chi_A(\lambda)$, then the algorithm can be terminated.
\end{remark}

\begin{remark}
  \label{remark:reduction-krylovGS}
  \Cref*{alg:jordan-krylov-basis} is usually executed 
  with multiple-precision arithmetic over integers or rational numbers.
   An increase in the number of arithmetic operations
  may cause an increase in  
  the number of digits of nonzero elements in the computed 
  {\JKbasis} or Jordan chains, which may make the computation slow.
  However, keeping a ``simpler'' form of {\KrylovGS},
  it is expected that the computation of {\JKelim} and Jordan chains
  may be more efficient, as follows.
  \begin{enumerate}
    \item If the input vectors have a simpler form, 
    then the computed {\JKbasis} may have a simpler form
    and its computation may be faster.
    \item If a {\JKbasis} is computed in a simpler form, 
    then the computation of the Jordan chains may be faster.      
  \end{enumerate}
  We call the output of Procedure~\ref{alg:reduction-krylovGS}, 
  a heuristic reduction of the {\KrylovGS} 
  which may be performed prior to the {\JKelim}, 
  a \emph{simpler form} of that set.
  This procedure uses a simple column reduction on the {\KrylovGS}.
  (See \Cref{sec:exp-20220526} for the effect of Procedure~\ref{alg:reduction-krylovGS} on suppressing the growth of coefficient sizes and reducing the computational cost.)
\end{remark}

\begin{algorithm}[H]
  \caption{Computing a {\JKbasis} of $\ker f(A)^{\lbar}$: a refined version}
  \label{alg:jordan-krylov-basis}
  \begin{algorithmic}[1]
    \Input{Matrices $A$ and $f(A)\in\K^{n\times n}$,
    a {\KrylovGS} $\cV=\bigcup_{\ell=1}^{\lbar}\cV^{(\ell)}$ of $\ker f(A)^{\lbar}$,

    $d=\deg(f(\lambda))$ and the multiplicity $m$ of $f(\lambda)$ in the characteristic polynomial of $A$}
    \Output{A {\JKbasis} $\cB=\bigcup_{\ell=1}^{\lbar}\cB^{(\ell)}$ of $\ker f(A)^{\lbar}$} 
    \State{(Optional) Call Procedure~\ref{alg:reduction-krylovGS} 
    with $f(A)$, 
    $\bigcup_{k=1}^{\lbar}\cV^{(k)}$}
    \State{Choose $\v\in\cV^{(\lbar)}$, \quad
    $\cV^{(\lbar)}\gets\cV^{(\lbar)}\setminus\{\v\}$,\quad
    $\cB^{(\lbar)}\gets\{\v\}$,\quad
    $m'\gets m-\lbar$}
    \If{$m'=0$} \label{line:jordan-krylov-basis:v}
      \State{\Return $\cB^{(\lbar)}$} 
      \Comment{Case \cref{eq:m=lbar}}
    \EndIf
      \State $S\gets[\cL_{A,d}(\v)]$,\quad
      $W\gets f(A)^{\lbar-1}S$
      \Comment{$W=f(A)^{\lbar-1}S$ can be reduced}
      \label{line:jordan-krylov-basis:SW}
   \For{$\ell=\lbar,\dots,2,1$}
      \label{line:jordan-krylov-basis:for-begin}
      \If{$\ell<\lbar$ and $\cV^{(\ell)}\ne\emptyset$}
        \State{(Optional) Call Procedure~\ref{alg:reduction-krylovGS} 
        with $f(A)$,
        $\bigcup_{k=1}^{\ell}\cV^{(k)}$}
      \EndIf
      \While {$\cV^{(\ell)} \ne \emptyset$}
        \label{line:jordan-krylov-basis:while-begin}
        \State Choose $\v\in \cV^{(\ell)}$, \quad
        $\cV^{(\ell)} \gets \cV^{(\ell)} \setminus \{\v\}$
        \State $\v'\gets f(A)^{\ell-1}\v$
        \label{line:jordan-krylov-basis:v'}
        \State Simultaneous column reduction of the rightmost column in the augmented matrix 
        
        \qquad
        $[W\mid\v']\longrightarrow[W\mid\r']$,\quad
        $[S\mid\v]\longrightarrow[S\mid\r]$
        \label{line:jordan-krylov-basis:elim}
        \If{$\r' \ne \0$}
        \Comment{$\r'\not\in\Span_K W$}
          \State $\cB^{(\ell)}\gets\cB^{(\ell)} \cup \{\r\}$, \quad $m'\gets m'-\ell$
          \If{$m'=0$} 
            \label{line:jordan-krylov-basis:ell}
            \State{\Return{$\cB=\cB^{(\lbar)}\cup\cB^{(\lbar-1)}\cup\cdots\cup\cB^{(\ell)}$}}
          \EndIf
          \State
          $S\gets[S\mid\cL_{A,d}(\r)]$,\quad 
          $W\gets[W\mid\cL_{A,d}(\r')]$
          \Comment{$W=f(A)^{\ell-1}S$ can be reduced}
        \ElsIf {$\r\ne \0$ and $\ell>1$}
          \State $\ell'\gets\rank_f\r$,\quad $\cV^{(\ell')}\gets\cV^{(\ell')}\cup\{\r\}$
        \EndIf
      \EndWhile
      \If{$\ell>1$}
        \State{$\cB^{(\ell-1)}\gets\emptyset$,\quad $S\gets f(A)S$}
        \Comment{For the next step in the loop}
        \label{line:jordan-krylov-basis:f(A)S}
      \EndIf
    \EndFor
    \State \Return{$\cB=\bigcup_{\ell=1}^{\lbar}\cB^{(\ell)}$}
  \end{algorithmic}
\end{algorithm}

\makeatletter
\renewcommand{\ALG@name}{Procedure}
\makeatother
  \begin{algorithm}[H]
    \caption{Reduction of {\KrylovGS} (See \Cref{remark:reduction-krylovGS})}
    \label{alg:reduction-krylovGS}
    
    \begin{algorithmic}[1]
       \Input{A matrix $f(A)\in\K^{n\times n}$, 
        a {\KrylovGS} $\bigcup_{k=1}^{\ell}\cV^{(k)}$of $\ker f(A)^{\ell}$,
      where $1\le\ell\le\lbar$}
      \Output{The reduced {\KrylovGS} $\bigcup_{k=1}^{\ell}\cV^{(k)}$} 
      \State{$\cT\gets\cV^{(\ell)}$, \quad $\cV^{(\ell)}\gets\emptyset$}
      \State{Reduce $[\cT]\longrightarrow[\cT']$ with a column reduction}
      \label{line:reduction-krylovGS:reduce}
      \While{$\cT'\ne\emptyset$}
        \State{Choose $\v\in\cT'$, \quad $\cT'\gets\cT'\setminus\{\v\}$}
        
        \If{$\v\ne\0$}
          \State{$\ell'\gets\rank_f\v$,\quad $\cV^{(\ell')}\gets\cV^{(\ell')}\cup\{\v\}$}
          \label{line:reduction-krylovGS:rank}
        \EndIf
      \EndWhile
      \State{\Return{$\bigcup_{k=1}^{\ell}\cV^{(k)}$}}
    \end{algorithmic}
  \end{algorithm}

\makeatletter
\renewcommand{\ALG@name}{Algorithm}
\makeatother

\subsection{An algorithm for computing generalized eigenvectors of $A$}
\label{sec:generalized-eigenspace}

\Cref{alg:single-jordan-chain} computes Jordan chains using the elements 
in a {\JKbasis} of $\ker f(A)^{\lbar}$. 
Note that $\lambda$ represents an arbitrary root of $f$, that is, the eigenvalue associated to $f$ up to conjugation.
Finally,
\Cref{alg:generalized-eigenspace} integrates
\Cref{alg:jordan-krylov-basis,alg:single-jordan-chain}
for computing the generalized eigenspace of $A$ associated to the roots of $f(\lambda)$.

\begin{algorithm}[h]
  \caption{Computing Jordan chains from a {\JKbasis} of $\ker f(A)^{\lbar}$}
  \label{alg:single-jordan-chain}
  \begin{algorithmic}[1]
    \Input{A matrix $f(A)\in\K^{n\times n}$, an irreducible factor $f(\lambda)\in\K[\lambda]$,

    a {\JKbasis} $\cB=\bigcup_{\ell=1}^{\lbar}\cB^{(\ell)}$ of $\ker f(A)^{\lbar}$}
    \Output{a set of Jordan chains $\Phi$ of $A$ associated to the roots of $f(\lambda)$}
    \State $\Phi\gets\emptyset$
    \State{Compute polynomials $\psi_f^{(1)}(\mu,\lambda),\dots,\psi_f^{(\lbar)}(\mu,\lambda)$}
    \label{line:single-jordan-chain:psi_f}
    \For{$\ell=\lbar,\lbar-1,\dots,1$}
    \label{line:single-jordan-chain:for}
      \While{$\cB^{(\ell)}\ne\emptyset$}
        \State Choose $\b\in\cB^{(\ell)}$, $\cB^{(\ell)}\gets\cB^{(\ell)}\setminus\{\b\}$
        \For{$k=\ell,\ell-1,\dots,2$}
          \State $\p^{(k)}\gets\psi_f^{(k)}(A,\lambda E)\b$
          \label{line:single-jordan-chain:p^k}
          \State{$\b\gets f(A)\b$}
          \label{line:single-jordan-chain:f(A)b}
        \EndFor
        \State $\p^{(1)}\gets\psi_f^{(1)}(A,\lambda E)\b$
        \State $\Phi\gets\Phi\cup\{\{\p^{(\ell)},\dots,\p^{(1)}\}\}$
      \EndWhile
    \EndFor
    \State \Return{$\Phi$}
  \end{algorithmic}
\end{algorithm}

\begin{algorithm}
  \caption{Computing the generalized eigenspace of $A$ associated to the roots of $f(\lambda)$}
  \label{alg:generalized-eigenspace}
  \begin{algorithmic}[1]
    \Input{A matrix $A\in\K^{n\times n}$, an irreducible factor $f(\lambda)\in\K[\lambda]$,

      a {\KrylovGS}  of $\ker f(A)^{\lbar}$ $\cV=\bigcup_{\ell=1}^{\lbar}\cV^{(\ell)}$,

      the multiplicity $m$ of $f(\lambda)$ in the characteristic polynomial of $A$}
    \Output{The generalized eigenspace (a set of Jordan chains) $\Phi$ of $A$
     associated to the roots of $f(\lambda)$ }
    \State{Calculate $f(A)$}
    \State Call \Cref{alg:jordan-krylov-basis} with $A$, $f(A)$, $\cV$, $d=\deg(f)$ and $m$
    for computing  $\cB = \bigcup_{\ell=1}^{\lbar} \cB^{(\ell)}$ 
    \State Call \Cref{alg:single-jordan-chain} with $f(A)$, $f(\lambda)$ and $\cB$
    for computing $\Phi$
    \State \Return{$\Phi$}
  \end{algorithmic}
\end{algorithm}

\subsection{Computing a {\KrylovGS} of $\ker f(A)^{\lbar}$ with the use of minimal annihilating polynomials}
\label{sec:ker-f(A)lbar-generators}

Let $\cE$ be a basis of $\K^n$ and let $\cP=\{\pi_{A,\e}(\lambda) \mid \e\in\cE\}$.
Since each minimal annihilating polynomial $\pi_{A,\e}(\lambda)$ can be expressed as 
\begin{equation}
  \label{eq:pi_a_ej}
  \pi_{A,\e}(\lambda)=f(\lambda)^{\ell_\e}g_\e(\lambda), \qquad \gcd(f,g_\e)=1,
\end{equation}
for some $\ell_{\e}\ge 0$ and $\pi_a(\lambda)=\lcm\{\pi_{A,\e}(\lambda)\mid\e\in\cE\}$,
it holds that 
$\lbar=\max\ \{ \ell_\e \mid \e \in \cE\}$ and $g_\e(A)\e\in \ker f(A)^{\lbar}$.
We define $\cE_f$ and $\cV$ as 
\[
  \cE_f=\{\e\in\cE\mid \pi_{A,\e}(\lambda)=f(\lambda)^{\ell_\e}g_\e(\lambda),\ 
  \ell_{\e}>0\},
  \quad
  \cV = \{ g_\e(A)\e \mid \e\in \cE_f\}.
\]
We have the following proposition.

\begin{proposition}
  \label{prop:ker-f(A)^l}
  It holds that $\ker f(A)^{\lbar}=\Span_{\K}\cL_A(\cV)$.
\end{proposition}
\begin{proof}
  Since $\cV\subset\ker f(A)^{\lbar}$, $\Span_{\K}\cL_A(\cV)\subset\ker f(A)^{\lbar}$ holds.  
  The opposite inclusion is shown as follows. 
  Let $\pi_A(\lambda)=f(\lambda)^{\lbar}g(\lambda)$, where 
  $g(\lambda)$ is relatively prime to $f(\lambda)$.
  Then, there exist 
  $a(\lambda),b(\lambda)\in K[\lambda]$ such
  that $a(\lambda)f(\lambda)^{\lbar}+b(\lambda)g(\lambda)=1$.
  Let $\v\in\ker f(A)^{\lbar}$, then, 
  $(a(A)f(A)^{\lbar}+b(A)g(A))\v=\v$ implies that $b(A)g(A)\v=\v$.
  Now, by expressing $b(A)\v=\sum_{\e\in\cE}c_{\e}\e$ with $c_{\e}\in\K$, 
  we have 
  $b(A)g(A)\v=g(A)(b(A)\v)=g(A)(\sum_{\e\in\cE}c_{\e}\e)=\v$.
  It implies that 
  $g(A)(\sum_{\e\in\cE}c_{\e}\e)=\sum_{\e\in\cE}c_{\e}g(A)\e=
  \sum_{\e\in\cE_f}c_{\e}g(A)\e$,
  because, for $\e\in\cE\setminus\cE_f$,
  $\pi_{A,\e}(\lambda)$ does not have $f(\lambda)$ as a factor and $g(A)\e=\0$.
  For $\e\in\cE_f$, by \cref{eq:pi_a_ej}, we have $g_{\e}(\lambda)\mid g(\lambda)$,
  thus there exists $h_{\e}(\lambda)\in K[\lambda]$ 
  such that $g(\lambda)=g_{\e}(\lambda)h_{\e}(\lambda)$.
  Therefore, 
  $\v=g(A)b(A)\v=\sum_{\e\in\cE_f}c_{\e}h_{\e}(A)g_{\e}(A)\e$,
  where $g_{\e}(A)\e\in\cV$.
  Furthermore, for $h_{\e}(\lambda)=\sum_{j=0}^k h_j\lambda^j$ we have
  $h_{\e}(A)g_{\e}(A)\e=\sum_{j=0}^k h_jA^jg_{\e}(A)\e \in L_A(g_{\e}(A)\e)$,
  which proves the claim.
\end{proof}

Proposition~\ref{prop:ker-f(A)^l} says that the set $\cV$ is a {\KrylovGS}
of $\ker f(A)^{\lbar}$.
In \cite{taj-oha-ter2018b}, an algorithm has been proposed for computing
all the minimal annihilating polynomials $\pi_{A,\e}(\lambda)$
of the element $\e$ of a basis $\cE$ of $\K^n$. 
Accordingly, the set $\cV$ is computable as shown 
in \Cref{alg:ker-f(A)-lbar-generators}.

\begin{algorithm}[H]
  \caption{Computing a {\KrylovGS} of $\ker f(A)^{\lbar}$}
  \label{alg:ker-f(A)-lbar-generators}
  \begin{algorithmic}[1]
    \Input{A matrix $A\in\K^{n\times n}$, an irreducible factor $f(\lambda)\in\K[\lambda]$,  
    a basis $\cE$ of $\K^n$, 
    
    the set of minimal annihilating polynomials $\{\pi_{A,\e}(\lambda)\mid\e\in\cE\}$}
    \Output{A {\KrylovGS} $\cV=\bigcup_{\ell=1}^{\lbar}\cV^{(\ell)}$ of $\ker f(A)^{\lbar}$}
    \State{$\lbar\gets\max_{\e\in\cE}
    \text{(multiplicity of $f(\lambda)$ in $\pi_{A,\e}(\lambda)$)}$}
    \For{$\ell=\lbar,\lbar-1,\dots,1$}
      \label{line:ker-f(A)-lbar-generators:for}
      \State{$\cV^{{(\ell)}}\gets\{g_{\e}(A)\e\mid\e\in\cE_f,\ell_{\e}=\ell,
      g_{\e}(\lambda)=\pi_{A,\e}(\lambda)/f(\lambda)^{\ell_{\e}}\}$}
      \label{line:ker-f(A)-lbar-generators:g_u(A)u}
    \EndFor
    \State{\Return{$\cV=\bigcup_{\ell=1}^{\lbar}\cV^{(\ell)}$}}
    \label{line:ker-f(A)-lbar-generators:return}
  \end{algorithmic}
\end{algorithm}

\begin{remark}
  In \Cref{alg:ker-f(A)-lbar-generators}, 
  $\cV^{(\ell)}=\{\v\in\cV\mid \rank_f\v=\ell\}$ holds.
\end{remark}

\Cref{alg:generalized-eigenspace-with-ker-f(A)-lbar-generators} presents \Cref{alg:generalized-eigenspace} with \Cref{alg:ker-f(A)-lbar-generators} incorporated.

\begin{algorithm}
  \caption{Computing the generalized eigenspace of $A$ associated to the roots of $f(\lambda)$ using minimal annihilating polynomials}
  \label{alg:generalized-eigenspace-with-ker-f(A)-lbar-generators}
  \begin{algorithmic}[1]
    \Input{A matrix $A\in\K^{n\times n}$, an irreducible factor $f(\lambda)\in\K[\lambda]$,
      a basis $\cE$ of $\K^n$,

      a set of minimal annihilating polynomials $\cP=\{\pi_{A,\u}(\lambda)\mid\u\in\cE\}$,

      the multiplicity $m$ of $f(\lambda)$ in the characteristic polynomial of $A$}
    \Output{The generalized eigenspace (a set of Jordan chains) $\Phi$ of $A$
     associated to the roots of $f(\lambda)$ }
    \State{Calculate $f(A)$}
    \State Call \Cref{alg:ker-f(A)-lbar-generators} with $A$, $f(\lambda)$, $\cE$ and $\cP$
    for computing  $\cV=\bigcup_{\ell=1}^{\lbar}\cV^{(\ell)}$ 
    \State Call \Cref{alg:jordan-krylov-basis} with $f(A)$, $\cV$ and $m$
    for computing  $\cB = \bigcup_{\ell=1}^{\lbar} \cB^{(\ell)}$ 
    \State Call \Cref{alg:single-jordan-chain} with $f(A)$, $f(\lambda)$ and $\cB$
    for computing $\Phi$
    \State \Return{$\Phi$}
  \end{algorithmic}
\end{algorithm}

\section{The time complexity of the algorithm}
\label{sec:time-complexity}

In this section, we discuss the time complexity of the algorithms presented in the previous sections in terms of the number of arithmetic operations over $\K$.

The computational complexity of the algorithm depends not only on the size of the input matrix but also on its structure as an eigenproblem, including the number of irreducible factors in the characteristic polynomial, their degrees, multiplicities, and the way all the minimal annihilating polynomials are factored into irreducible factors.
Thus, the analysis of computational complexity for a general matrix will become highly complicated.
Therefore, we treat a special but typical case with the following assumptions.

\begin{assumption}
  \label{assumption:time-complexity}
  Let $A\in\K^{n\times n}$ be the input matrix with
  $\chi_A(\lambda)=f(\lambda)^{\lbar+\nu}g(\lambda)$ and 
  $\pi_A(\lambda)=f(\lambda)^{\lbar}g(\lambda)$, 
  where $f(\lambda)\in\K[\lambda]$ and $\deg f(\lambda)=d$.
  Assume the following.
  \begin{enumerate}
    \item For the eigenvalues that are roots of $f(\lambda)$ and on which we focus, there is one Jordan chain of length $\lbar$ and $\nu$ linearly independent eigenvectors associated to that eigenvalue.
    \item The polynomial $g(\lambda)$ is square-free in $\K[\lambda]$.
    \item 
    Let 
    \begin{equation}
      \label{eq:rt}
      \begin{split}
        r &= \#\{\e\in\cE\mid \mbox{$f(\lambda)$ divides $\pi_{A,\e}(\lambda)$}\},\\
        t &= \#\{\e\in\cE\mid \mbox{$f(\lambda)$ divides $\pi_{A,\e}(\lambda)$ and $\deg(g_{\e}(\lambda))>0$ }\}.        
      \end{split}
    \end{equation}
  \end{enumerate}
\end{assumption}

\begin{assumption}
  \label{assumption:minimal-annihilating-polynomial}
  Assume that, for $\e\in\cE$, the minimal annihilating polynomial $\pi_{A,\e}(\lambda)$ is given.
  As a result, the matrices $f(A)$ and $g_{\e}(A)$ are also given.
\end{assumption}

\Cref{assumption:time-complexity} is about the structure of the input matrix.
Note that, by the assumption, we have $r\ge d(\lbar+\nu)$, which shows that
$\max\{d,\nu,\lbar,r\}=r$, $d\lbar\le r\le n$ and $d\nu\le r\le n$.
\Cref{assumption:minimal-annihilating-polynomial} means that the complexity of
calculating the minimal annihilating polynomials is not considered in this analysis.
For calculating the minimal annihilating polynomials,
the choice of algorithm, whether deterministic or probabilistic, 
can affect the overall computational complexity estimation. 
In this paper, we will focus on the complexity estimation of algorithms for 
calculating general eigenvectors, excluding the estimation of the algorithm for calculating the minimal annihilating polynomials, which can be found in our previous paper
 (\cite{taj-oha-ter2018b}).
Furthermore, since $f(A)$ and $g_{\e}(A)$ are obtained during the computation of 
the minimal annihilating polynomials, they will be excluded from this evaluation.

\begin{remark}
  When computing the Krylov generating sets for $\ker f(A)^{\lbar}$, it is not always necessary to compute $g_{\e}(\lambda)$ precisely. Instead, $g(\lambda)$ can also be used, which may make the complexity analysis simpler. However, our approach aims to compute the Krylov generating sets without including redundant factors in $g(\lambda)$, to reduce the length of the components of the vectors appearing during the computation and thereby improve computational efficiency (see also \Cref{remark:reduction-krylovGS}).
\end{remark}

Based on the above assumptions, we first present an estimate of the complexity of the overall algorithm as a conclusion.

\begin{theorem}
  \label{thm:time-complexity}
  Under \Cref{assumption:time-complexity,assumption:minimal-annihilating-polynomial}, the complexity of 
  \Cref{alg:generalized-eigenspace}  is 
  $O(n^2r^2)$ 
  plus the complexity of calculating $f(A)$.
\end{theorem}

For the proof of \Cref{thm:time-complexity}, we estimate the complexity of each step as follows.
First, we note that \Cref{alg:ker-f(A)-lbar-generators} is essentially included in the 
computation of the minimal annihilating polynomials, as follows.

\begin{remark}
  \label{rem:omit}
  In the algorithm for computing the minimal annihilating polynomials, a pseudo minimal 
  annihilating
  polynomial is computed as $\pi'_{A,\e}(\lambda)=f(\lambda)^{\ell'_\e}g'_{\e}(\lambda)$
  (see \Cref{rem:pseudo-minimal-annihilating-polynomial}).
  Then, after computing $g'_{\e}(A)\e$, it is tested whether $f(A)^{\ell'_\e}g'_{\e}(A)\e=\0$.
  If it is satisfied, $g'_{\e}(A)\e$ can be used as an element in the {\KrylovGS}, 
  hence the computation in \Cref{alg:ker-f(A)-lbar-generators} can be omitted
  (see \cite{taj-oha-ter2018}).
\end{remark}
For the complexity of calculating
the minimal annihilating polynomials, see 
\Cref{sec:time-complexity-minimal-annihilating-polynomials}.
We estimate the complexity of \Cref{alg:jordan-krylov-basis,alg:reduction-krylovGS,alg:single-jordan-chain}.

\begin{lemma}
  \Cref{alg:jordan-krylov-basis} is performed in $O(n^2r^2)$.
\end{lemma}
\begin{proof}
  The complexity of each step is as follows:
  \begin{itemize}
    \item Line~\ref{line:jordan-krylov-basis:SW}: The columns of $S$ are calculated as repeating calculation of $A\v$ for $d-1$ times, which costs $O(dn^2)$.
    For calculating $f(A)^{\lbar-1}S$, the multiplication $f(A)S$ is performed in $O(dn^2)$ since
    the size of $S$ is $n\times d$.
    Then, this multiplication is repeated $\lbar-1$ times, which costs $O(d\lbar n^2)$.
    If the reduction of $W$ is performed, its cost is $O(d^2 n)$. 
    Thus, this step is dominated by $O(d\lbar n^2)$. 
    Furthermore, since we have $d\lbar\le r$, this is $O(n^2r)$.
    \item Line~\ref{line:jordan-krylov-basis:while-begin}: The while-loop, together with the for-loop in line~\ref{line:jordan-krylov-basis:for-begin}, is repeated $r$ times, which is equal to $|\cV^{(1)}|+\cdots+|\cV^{(\lbar)}|$.
    In line~\ref{line:jordan-krylov-basis:v'}, the vector $\v'$ is calculated in $O(\ell n^2)$, which is bounded by $O(\lbar n^2)$.
    In line~\ref{line:jordan-krylov-basis:elim}, the simultaneous column reduction is performed in $O(n \times (\mbox{the number of columns in $W$ and $S$}))$.
    The number of elements in the {\JKbasis} of $\ker f(A)^{\lbar}$ is one for rank $\lbar$, and $\nu$ for rank $1$.
    Thus, the number of columns in $W$ and $S$ is bounded by $d(\lbar+\nu)\le r$, 
    and the simultaneous column reduction is performed in $O(nr)$.
    Therefore, the total complexity of the while-loop, together with the for-loop, is
    $O(dnr(n\lbar+r))$.
    Since $r\le n$, this is $O(dnr(n\lbar))=O(dn^2r\lbar)$.
    Then, by using $d\lbar\le r$, we have $O(dn^2r\lbar)=O(n^2r^2)$.
    \item 
    Line~\ref{line:jordan-krylov-basis:f(A)S} is performed in $O(dn^2)$, thus this step is performed in $O(d\lbar n^2)$ in total over the for-loop in Line~\ref{line:jordan-krylov-basis:for-begin}.
    Furthermore, since we have $d\lbar\le r$, this is $O(n^2r)$.
  \end{itemize}
  Summarizing above, the complexity of \Cref{alg:jordan-krylov-basis} is 
  $O(n^2r^2)$, which proves the lemma. 
\end{proof}

\begin{lemma}
  Procedure~\ref{alg:reduction-krylovGS} is performed in $O(n^2r^2)$.
\end{lemma}
\begin{proof}
  The number of elements in the {\KrylovGS} of $\ker f(A)^{\lbar}$ is $r$, 
  thus, line~\ref{line:reduction-krylovGS:reduce} is performed in $O(nr^2)$.
  In line~\ref{line:reduction-krylovGS:rank}, the complexity of calculating 
  $\rank_f\v$ is estimated as follows. 
  The complexity of calculating $f(A)\v$ is $O(n^2)$, and the repeated calculation of $f(A)\v$ for $\rank_f\v$ is performed in $O(\lbar n^2r)$ for all $\v$ in $\cV^{(\ell)}$.
  Thus, the complexity of the whole algorithm is $O(nr(r+\lbar n))$.
  Furthermore, since we have $\lbar<r$, this is $O(nr(r+rn))=O(n^2r^2)$, 
  which proves the lemma.
\end{proof}

\begin{lemma}
  \Cref{alg:single-jordan-chain} is performed in $O(n^2r^2)$.
\end{lemma}
\begin{proof}
  Each step is performed as follows:
  \begin{itemize}
    \item Line \ref{line:single-jordan-chain:psi_f}: 
    The complexity of calculating 
    $\psi_f^{(\ell)}(\mu,\lambda)$ is estimated as follows. 
    Note that $\psi_f^{(\ell)}(\mu,\lambda)$ is regarded as a polynomial in 
    $\K[\lambda][\mu]$.
    We first estimate the complexity with respect to $\mu$.
    For $\ell=1$, 
    since $\deg_\mu(f(\lambda)-f(\mu))=d$ and $\deg_\mu(\lambda-\mu)=1$,
    $(f(\lambda)-f(\mu))/(\lambda-\mu)$ is calculated in $O(d)$.
    For $\ell>1$, 
    $\psi_f^{(\ell)}(\mu,\lambda)$ is calculated recursively by 
    $\psi_f^{(\ell)}(\mu,\lambda)=\psi_f^{(\ell-1)}(\mu,\lambda)\psi_f^{(1)}(\mu,\lambda)$.
    This computation is performed in $K[\lambda][\mu]$. 
		Furthermore, since $f(\lambda)$ is the minimal polynomial of the eigenvalue of $A$ 
		currently under consideration, the actual computation takes place in
		$\left(K[\lambda]/\langle f(\lambda) \rangle\right)[\mu]$,
    which involves a reduction by $f(\lambda)$. 
    Since $\deg_\mu(\psi_f^{(\ell-1)}(\mu,\lambda))=(d-1)(\ell-1)$ and
    $\deg_\mu(\psi_f^{(1)}(\mu,\lambda))=d-1$,
    $\psi_f^{(\ell-1)}(\mu,\lambda)\psi_f^{(1)}(\mu,\lambda)$ is calculated in 
    $O(d^2\ell)$.
    Thus, $\psi_f^{(1)}(\mu,\lambda),\dots,\psi_f^{(\lbar)}(\mu,\lambda)$ are 
    calculated in $O(d^2\lbar^2)$.
    For the coefficients in $\K[\lambda]/\langle f(\lambda) \rangle$, 
    the coefficients are calculated as the product of two polynomials in $\K[\lambda]$ of degree
    less than or equal to $d$, which costs $O(d^2)$, followed by the division 
    of a polynomial of degree less than or equal to $d^2$ by $f(\lambda)$,
    which costs $O(d^2)$.
    Therefore, total cost is $O(d^4\lbar^2)$.
    \item Line \ref{line:single-jordan-chain:for}: By the assumption, the for-loop is repeated for $\ell=\lbar,1$.
    In line \ref{line:single-jordan-chain:p^k}, 
    $\p^{(k)}$ is calculated using Horner's rule with matrix-vector multiplications in 
    $O(n^2dk)$, since $\deg(\psi_f^{(k)})=(d-1)k$.
    In line \ref{line:single-jordan-chain:f(A)b}, the multiplication $f(A)\b$ is performed in 
    $O(n^2)$.
    Thus, calculating a Jordan chain of length $\lbar$ and $1$ is performed 
    in $O(n^2d\lbar^2)$ (since $\p^{(k)}$ is calculated for $k=\lbar,\dots,2$)
    and $O(n^2 d)$, respectively,
    and the complexity of calculating the Jordan chains is bounded by $O(n^2d\lbar^2)$.
  \end{itemize}
  Summarizing the above, the complexity of \Cref{alg:single-jordan-chain} is 
  $O(d^4\lbar^2+n^2d\lbar^2)$.
  Furthermore, since we have $\lbar<d\lbar\le r$, this is $O((d^2+n^2)r^2)$, and,
  with $d<n$, this is $O(n^2r^2)$,
  which proves the lemma.
\end{proof}

\paragraph{A proof of \Cref{thm:time-complexity}} 
Summarizing the above, we see that each step in Lines 2--4 of \Cref{alg:generalized-eigenspace} is performed in $O(n^2r^2)$, which proves the theorem.
\qed
\begin{remark}
  The exclusion of the computational cost of $f(A)$ in the complexity estimate of 
  \Cref{thm:time-complexity} is based on the following reason. In this paper, for example, in the experiments, the computation of $f(A)$ is performed using classical matrix multiplication. 
  However, considering situations where one might use a ``fast multiplication algorithm'' with multiplication exponent 
  $\omega$ (\cite{wil-xu-xu-zhou2024}), we present a complexity estimate that excludes the cost of computing $f(A)$.
\end{remark}

\subsection{The time complexity of calculating the minimal annihilating polynomials}
\label{sec:time-complexity-minimal-annihilating-polynomials}

We have proposed efficient algorithms for calculating the minimal annihilating polynomials (\cite{taj-oha-ter2018b}). 
Here, we evaluate the arithmetic complexities of 
the deterministic and probabilistic algorithms.
We estimate the complexities of the algorithms under 
\Cref{assumption:time-complexity,assumption:minimal-annihilating-polynomial}.
In \Cref{assumption:time-complexity}, assume that $g(\lambda)$ is square-free and has $q$  irreducible factors in $\K[\lambda]$, satisfying $g(\lambda)=g_1(\lambda)\cdots g_q(\lambda)$. 
Under this assumption, we give an overview of the algorithm and estimate its complexity.

Before the estimation of the complexity, we 
briefly describe the algorithm for 
calculating the minimal annihilating polynomial.
Let $\cE$ be a basis in $\K^n$ and $\e\in\cE$. 
For calculating the minimal annihilating polynomial $\pi_{A,\e}(\lambda)$,
we first calculate $g(A)\e$. 
If $g(A)\e\ne\0$, then $\pi_{A,g(A)\e}(\lambda)$ is a power of $f(\lambda)$.
Its exponent $\ell$ is what we estimate, and it is equal to the exponent of $f(\lambda)$ in 
$\pi_{A,\e}(\lambda)$.
Next, for $f(A)^\ell\e$, we determine which of $g_1(\lambda),\dots,g_q(\lambda)$ 
are included as a factor of $\pi_{A,\e}(\lambda)$.
For example, if $f(A)^{\ell}\prod_{i=1}^{q-1} g_i(A)\e\ne\0$, then $g_q(\lambda)$ is a factor of $\pi_{A,\e}(\lambda)$.
In this way, we determine if each $g_i(\lambda)$ is a factor of $\pi_{A,\e}(\lambda)$.
The deterministic algorithm is summarized as \Cref{alg:minimal-annihilating-polynomials}.

\begin{remark}
  \label{rem:pseudo-minimal-annihilating-polynomial}
  The probabilistic algorithm is based on the following idea.
  Take a random row vector $\v'\in\K^n$ and
  replace the computation of the matrix $h(A)$ by that of the vector $\v'h(A)$
  to reduce the amount of calculation.
  The probabilistic algorithm is summarized in \cite[Algorithm 4]{taj-oha-ter2018}.
  We refer to the polynomial computed by the probabilistic method as a ``pseudo minimal annihilating polynomial,'' denoted by $\pi'_{A,\e}(\lambda)$.
\end{remark}

\begin{algorithm}
  \caption{Computing the minimal annihilating polynomials (deterministic version)}
  \label{alg:minimal-annihilating-polynomials}
  \begin{algorithmic}[1]
    \Input{A matrix $A\in\K^{n\times n}$, an irreducible factor $f(\lambda)\in\K[\lambda]$,
    the characteristic polynomial $\chi_A(\lambda)=f(\lambda)^{\lbar+\nu}g_1(\lambda)\cdots g_q(\lambda)$, a basis $\cE$ of $\K^n$}
    \Output{The set of minimal annihilating polynomials $\cP=\{\pi_{A,\e}(\lambda)\mid \e\in\cE\}$}
    \State{$\cP\gets\emptyset$}
    \label{line:minimal-annihilating-polynomials:g}
    \For{$\e\in\cE$}
    \label{line:minimal-annihilating-polynomials:for-e}
      \State{$\b\gets g(A)\e$, \quad $\b_0\gets\e$, \quad $\pi_{A,\e}(\lambda)\gets 1$}
      \label{line:minimal-annihilating-polynomials:g(A)e}
      \If{$\b\ne\0$}
        \label{line:minimal-annihilating-polynomials:if}
        \Comment{$f(\lambda)$ divides $\pi_{A,\e}(\lambda)$}
        \For{$i=1,2,\dots,\lbar$}
        \label{line:minimal-annihilating-polynomials:for}
          \State{$\b\gets f(A)\b$,\quad $\b_0\gets f(A)\b_0$}
          \label{line:minimal-annihilating-polynomials:b}
          \If{$\b=\0$}
            \State{$\pi_{A,\e}(\lambda)\gets f(\lambda)^{i}$}
            \State{\textbf{break}}
          \EndIf
          \label{line:minimal-annihilating-polynomials:endif}
        \EndFor
      \EndIf
        \State{$i_{\mathrm{max}}\gets q$}
        \While{$i_{\mathrm{max}}\ge 1$ and $\b_0\ne\0$}
        \label{line:minimal-annihilating-polynomials:while}
        \Comment{$f(\lambda)$ divides $\pi_{A,\e}(\lambda)$ and $\deg(g_{\e}(\lambda))>0$}
          \State{$\b\gets\b_0$}
          \For{$i=1,2,\dots,i_{\mathrm{max}}$}
          \label{line:minimal-annihilating-polynomials:for2}
            \State{$\b\gets g_i(A)\b$}
            \label{line:minimal-annihilating-polynomials:b1}
            \If{$\b=\0$}
            \State{$\b_0\gets g_i(A)\b_0$,\quad 
             $\pi_{A,\e}(\lambda)\gets \pi_{A,\e}(\lambda) g_i(\lambda)$,\quad
             $i_{\mathrm{max}}\gets i-1$}
             \label{line:minimal-annihilating-polynomials:b0}
             \State{\textbf{break}}
            \EndIf
          \EndFor
        \EndWhile 
        \State{$\cP\gets\cP\cup\{\pi_{A,\e}(\lambda)\}$}
    \EndFor
    \State{\Return{$\cP$}}
  \end{algorithmic}
\end{algorithm}

\begin{proposition}
  \label{prop:time-complexity-minimal-annihilating-polynomials}
  Let $d_i=\deg g_i(\lambda)$ for $i=1,\dots,q$,
  and $r$ and $t$ be as in \cref{eq:rt}.
  Then, 
  \Cref{alg:minimal-annihilating-polynomials} computes
  those $\pi_{A,\e}(\lambda)$ that have $f(\lambda)$ as a factor
  in $O(n^3q^2r)$ operations, and the probabilistic algorithm requires $O(n^2q^2r)$ operations.
\end{proposition}
\begin{proof}
  The complexity of each step is as follows:
  \begin{enumerate}
    \item Line~\ref{line:minimal-annihilating-polynomials:g(A)e}: The calculation of $g(A)\e$ is performed in $O(n^2(n-(\lbar+\nu)d))=O(n^3)$.
    \item The for loop in Line~\ref{line:minimal-annihilating-polynomials:for}: The for loop is repeated $\lbar$ times. 
    Line~\ref{line:minimal-annihilating-polynomials:b} is performed in $O(dn^2)$.
    Thus, the total complexity of the for-loop is $O(dn^2\lbar)$.
    \item The While loop in Line~\ref{line:minimal-annihilating-polynomials:while}, together with the for-loop in Line~\ref{line:minimal-annihilating-polynomials:for2}: 
    The While loop is repeated at most $\frac12 q^2$ times.
    Lines~\ref{line:minimal-annihilating-polynomials:b1} and
    \ref{line:minimal-annihilating-polynomials:b0} are performed in 
    $O(d_i n^2)$, with $\sum_i d_i = n-d(\lbar+\nu)$.
    Thus, the total complexity of the loop is $O(n^3q^2)$.
  \end{enumerate}
  Furthermore, the For loop in 
  Line~\ref{line:minimal-annihilating-polynomials:for} is executed for 
  $r$ times for calculating $\pi_{A,\e}(\lambda)$ which has $f(\lambda)$ as a factor,
  and the While loop in Line~\ref{line:minimal-annihilating-polynomials:while} is executed for $t$ times
  for calculating $\pi_{A,\e}(\lambda)$ which has $f(\lambda)$ as a factor 
  and $\deg(g_{\e}(\lambda))>0$.
  Thus, the computing time for the While loop in $O(n^3q^2r)$ dominates the whole computing steps, which proves the proposition.
  For a probabilistic algorithm, the exponent of $n$ is reduced to $2$, which proves the proposition.
\end{proof}

\begin{remark}
  In \Cref{alg:minimal-annihilating-polynomials}, the for loop in 
  Line~\ref{line:minimal-annihilating-polynomials:for} can be parallelized 
  and calculated efficiently.
  Furthermore, during the calculation of the minimal annihilating polynomials,
  the {\KrylovGS} can be calculated simultaneously.
\end{remark}

\section{An example of computation}
\label{sec:examples}

We give an example for illustration, in which $\K=\Q$.
Note that, in the example, Procedure~\ref{alg:reduction-krylovGS} 
is not performed.
Let $A$ be a square matrix of order $10$ defined as 
\[
A=
\left(
\begin{array}{rrrrrrrrrr}
    5 & -5 & 6 & -9 & 5 & 0 & 0 & -4 & 5 & -6 \\
    -14 & 11 & -9 & 39 & -2 & -2 & 6 & 16 & -10 & 12 \\
    -5 & 5 & -6 & 9 & -5 & 1 & 0 & 5 & -5 & 5 \\
    5 & 2 & 1 & 7 & 7 & -4 & 6 & 3 & 5 & 2 \\
    -5 & -9 & 9 & -9 & -1 & 3 & -5 & -7 & -5 & -9 \\
    5 & 2 & -4 & -2 & 5 & -5 & 5 & -1 & 5 & 2 \\
    5 & 9 & -14 & 0 & -3 & -4 & 3 & 4 & 5 & 9 \\
    -5 & -9 & 4 & -23 & -8 & 7 & -11 & -11 & -5 & -9 \\
    0 & 8 & -6 & 16 & 2 & -4 & 6 & 7 & 0 & 9 \\
    4 & -7 & 4 & -25 & -3 & 3 & -6 & -11 & 0 & -8
\end{array}
\right).
\]
We have $\chi_A(\lambda)=f_1(\lambda)^4 f_2(\lambda)$, where
$f_1(\lambda)=\lambda^2+\lambda+5$ and $f_2(\lambda)=\lambda^2+\lambda+4$,
and $m_1=4$ and $m_2=1$.
Let $\cE=\{\e_1,\e_2,\dots,\e_{10}\}$ be the standard basis of $\K^{10}$.
Then, with the method we have proposed (\cite{taj-oha-ter2018}), 
the minimal annihilating polynomials are calculated as
\begin{align*}
  \pi_{A,\e_j}(\lambda) &=
  \begin{cases}
    f_1(\lambda) & (j=2,9) \\
    f_1(\lambda)f_2(\lambda) & (j=1,10) \\
    f_1(\lambda)^2f_2(\lambda) & (j=3) \\
    f_1(\lambda)^3 & (j=4,5,7) \\
    f_1(\lambda)^3f_2(\lambda) & (j=6,8)
  \end{cases}
  .
\end{align*}
We see that $\lbar_1=3$ and $\lbar_2=1$.

\subsection{Computing the generalized eigenspace associated to the roots of $f_1(\lambda)$}

  Let us compute the Jordan chains through a {\JKbasis} of $\ker f_1(A)^3$.
  First,
  \Cref{alg:ker-f(A)-lbar-generators} computes
  \[
    \cV_1^{(3)} = \{\v_{1,4},\v_{1,5},\v_{1,6},\v_{1,7},\v_{1,8}\},\quad
    \cV_1^{(2)} = \{\v_{1,3}\},\quad 
    \cV_1^{(1)} = \{\v_{1,1},\v_{1,2},\v_{1,9},\v_{1,10}\},
  \]
  as outputs, 
  where $\v_{1,j}=f_2(A)\,\e_j$ ($j=1,3,6,8,10$) and $\bm{e}_j$ ($j=2,4,5,7,9$).
  
  Next, \Cref{alg:jordan-krylov-basis} computes a {\JKbasis}
  $\cB_1=\cB_1^{(3)}\cup\cB_1^{(2)}\cup\cB_1^{(1)}$ 
  as follows. 
  Let $m=m_1=4$ and $\ell=\ell_1=3$.
  For $\ell=3$, 
  choose $\v_{1,4} \in \cV^{(3)}$
  and assign
  \begin{equation}
    \label{eq:example-2-f1-B3}
    \cB_1^{(3)} \gets \{\bm{v}_{1,4}\}, \quad
    S \gets [\cL_{A,d}(\{\v_{1,4}\})]=[\v_{1,4},A\v_{1,4}], \quad
    W \gets f_1(A)^2S,\quad
    m\gets m-\lbar=1.
  \end{equation}
  Then, we see that 
  \begin{align*}
    \v'_{1,5} = \v'_{1,4}+(2/5)A\v'_{1,4}, \quad \v'_{1,6} =\v'_{1,4}, \quad
    \v'_{1,7} = \v'_{1,4}+(1/5)A\v'_{1,4}, \quad \v'_{1,8} =(1/5)A\v'_{1,4},
  \end{align*}
  where $\v'_{1,i}=f_1(A)^2\v_{1,i}$.
  Thus, computation of $\cB_1^{(3)}$ is finished and 
  $S$ and $W$ remain the same,
  and let
  $\v_{1,11} = 5\v_{1,5}-5\v_{1,4}-2A\v_{1,4}$, $\v_{1,12} = \v_{1,6}-\v_{1,4}$, 
  $\v_{1,13} = 5\v_{1,7}-5\v_{1,4}-A\v_{1,4}$, and $\v_{1,14} = 5\v_{1,8}-A\v_{1,4}$.
  Since the ranks of $\v_{1,11}$, $\v_{1,12}$, $\v_{1,13}$ and $\v_{1,14}$ are equal to 2, 
  $\cV_1^{(2)}$ is updated to 
  $\cV_1^{(2)} \gets \cV_1^{(2)} \cup \{\v_{1,11},\v_{1,12},\v_{1,13},\v_{1,14}\}$.
  
  For $\ell=2$, assign
  \begin{equation}
    \label{eq:example-2-f1-B2-S}
    \cB_1^{(2)}\gets\emptyset,\quad
    S\gets f_1(A)S=\{f_1(A)\v_{1,4},Af_1(A)\v_{1,4}\},
  \end{equation}
  and let $W$ remain the same.
  For all $\v_1'=f_1(A)\v_1$ with $\v_1\in\cV_1^{(2)}$,
  there exists a simultaneous column reduction as
  $[W\mid \v_1']\rightarrow[W\mid\0]$ and
  $[S\mid\v_1]\rightarrow[S\mid\r]$, thus we have
  $\v_{1,3}\rightarrow\v_{1,15}$,
  $\v_{1,11}\rightarrow\v_{1,16}$,
  $\v_{1,12}\rightarrow\v_{1,17}$,
  $\v_{1,13}\rightarrow\v_{1,18}$,
  $\v_{1,14}\rightarrow\v_{1,19}$.
  Since the ranks of $\v_{1,15},\v_{1,16},\v_{1,17},\v_{1,18}$ and $\v_{1,19}$ are equal to $1$, 
  $\cV_1^{(1)}$ is updated to 
  $\cV_1^{(1)}\gets\cV_1^{(1)}\cup \{\v_{1,15},\v_{1,16},\v_{1,17},\v_{1,18},\v_{1,19}\}$.
  As a result, $\cB_1^{(2)}=\emptyset$ and $W$ remain the same.
  
  For $\ell=1$, 
  since $m\ge\ell$, there exists a vector 
  $\v_1\in\cV_1^{(1)}$ 
  such that a column reduction of the rightmost column outputs 
    $[W\mid \v_1]\rightarrow[W\mid\r\ne\0]$.
    The column reduction of $\v_{1,1}$ with respect to $W$ yields 
    \[
      \r={}^t(0,0,0,0,0,0,0,0,-1,0),
    \]
    assign $\cB_1^{(1)}\gets\{\r\}$ and $m\gets m-\ell=0$.
    Then, the algorithm terminates.
  
    As a consequence, a Jordan-Krylov basis $\cB_1$ is computed as 
  \begin{equation}
    \cB_1=\cB_1^{(3)}\cup\cB_1^{(1)},\quad
    \cB_1^{(3)}=\{\v_{1,4}\},\quad
    \cB_1^{(1)}=\{\r\}.
  \end{equation}
  
  Finally, \Cref{alg:single-jordan-chain} computes a set of Jordan chains.
  For $i=1,2,3$, $\psi^{(k)}_{f_1}(\mu,\lambda)$ is computed as
  \begin{align*}
    \psi^{(1)}_{f_1}(\mu,\lambda) &= \mu+\lambda+1, \qquad 
    \psi^{(2)}_{f_1}(\mu,\lambda) = \mu^2+(2\lambda+2)\mu+\lambda-4,\\
    \psi^{(3)}_{f_1}(\mu,\lambda) &= \mu^3+(3\lambda+3)\mu^2+(3\lambda-12)\mu+
  (-4\lambda-9).
  \end{align*}
  For $\v_{1,4}\in\cB_1^{(3)}$, the Jordan chain 
  $\{
    \p^{(3)}_1(\lambda,\v_{1,4}),\p^{(2)}_1(\lambda,\v_{1,4}),\p^{(1)}_1(\lambda,\v_{1,4})
  \}$
  of length $3$ is computed as
  \begin{align*}
    \p_1^{(3)}(\lambda,\v_{1,4}) &= \psi_{f_1}^{(3)}(A,\lambda E)\v_{1,4} \\
    &= \lambda\,{}^t(57,-60,-57,8,36,-3,-66,6,-30,3) \\
    &\quad + {}^t(205,-755,-205,-121,150,54,6,401,-307,455),\\
    \p_1^{(2)}(\lambda,\v_{1,4}) &= \psi_{f_1}^{(2)}(A,\lambda E)f_1(A)\v_{1,4} \\
    &= \lambda\,{}^t(11,32,-11,35,-78,35,78,-78,24,-43) \\
    &\quad + {}^t(-175,225,175,-49,-191,46,286,-96,126,-50),\\
    \p^{(1)}_1(\lambda,\v_{1,4}) &= \psi_{f_1}(A,\lambda E)f_1(A)^2\v_{1,4} \\
    &= 19(\lambda\,{}^t(0,1,0,1,-2,1,2,-2,1,-1) + {}^t(-5,11,5,1,-7,1,7,-7,6,-6)).
  \end{align*}
  For $\r\in\cB^{(1)}_1$, the Jordan chain 
  $\{
    \p^{(1)}_1(\lambda,\r)
  \}$
  of length $1$ is computed as
  \[
    \p^{(1)}_1(\lambda,\r) = \psi_{f_1}(A,\lambda E)\r 
    = \lambda\,{}^t(0,0,0,0,0,0,0,0,-1,0) + {}^t(-5,10,5,-5,5,-5,-5,5,-1,0),
  \]
  that is, $\p^{(1)}_1(\lambda,\r)$ is an eigenvector.

  \begin{remark}
    In \cref{eq:example-2-f1-B3}, $m$ is updated to $m=1$, then we see that
    $\#\cB_1^{(3)}=1$, $\cB_1^{(2)}=\emptyset$ and $\#\cB_1^{(1)}=1$.
    Choose a vector $\v_{1,1}$ from 
    $\cV_1^{(1)}=\{\v_{1,1},\v_{1,2},\v_{1,9},\v_{1,10}\}$.
    Then, the result $\r$ of the rightmost column reduction
    $[W\mid \v_{1,1}]\rightarrow[W\mid \r]$
    of $\v_{1,1}$ with respect to $W$ is not equal to zero.
    Thus, we set $\cB_1^{(1)}=\{\r\}$.
    Accordingly, a {\JKbasis} is given as 
    $\cB_1^{(3)}\cup\cB_1^{(1)}=\{\v_{1,4}\}\cup\{\v_{1,1}\}$.
    
  \end{remark}

  \subsection{Computing the generalized eigenspace associated to the roots of $f_2(\lambda)$}
  
  Let us compute a Jordan chain through a {\JKbasis} of $\ker f_2(A)$.
  First, 
  \Cref{alg:ker-f(A)-lbar-generators} computes
  \begin{equation}
    \label{eq:example-2-f2-V1}
    \cV_2=\cV_2^{(1)}=
    \{\v_{2,1},\v_{2,3},\v_{2,6},\v_{2,8},\v_{2,10}\},
  \end{equation}
  where $\v_{2,j}=f_1(A)\e_j$ ($j=1,10$), $f_1(A)^2\e_j$ ($j=3$), 
  $f_1(A)^3\e_j$ ($j=6,8$).
  Next, with \Cref*{alg:jordan-krylov-basis},
  choose $\v_{2,1}\in\cV_2^{(1)}$ and 
  let us assign $\cB_2^{(1)}\gets\{\bm{v}_{2,1}\}$, then
  $\lbar_2=m_2$ implies that $\cB_2=\cB_2^{(1)}$.
  Thus, as a {\JKbasis} of $\ker f_2(A)$, we have
  $
    \cB_2=\cB_2^{(1)}=\{\v_{2,1}\}.
  $
  Finally, with \Cref{alg:single-jordan-chain}, 
  using 
  $\psi^{(1)}_{f_2}(\mu,\lambda)=\psi_{f_2}(\mu,\lambda)=\mu+\lambda+1$
  and $\{\v_{2,1}\}\in\cB_2$, the 
  eigenvector 
  $\p_2^{(1)}(\lambda,\v_{2,1})$
  is computed as
  \[
    \p_2^{(1)}(\lambda,\v_{2,1}) = \psi_{f_2}(A,\lambda E)\v_{2,1} 
    = \lambda\,{}^t(1,0,0,0,0,0,0,0,-1,0) + {}^t(1,-4,0,0,0,0,0,0,-1,4).
  \]

\begin{remark}
  If Procedure~\ref{alg:reduction-krylovGS} is applied to
  $\cV_2^{(1)}$ in \cref{eq:example-2-f2-V1}, linearly independent vectors
  are obtained as 
  \[
    \v_{2,1}={}^t(1,0,0,0,0,0,0,0,-1,0),\quad \v_{2,6}={}^t(0,1,0,0,0,0,0,0,0,-1).
  \]
  It is expected that the use of 
  Procedure \ref{alg:reduction-krylovGS} 
  makes the computation more efficient.

\end{remark}

We see that the matrix $A$ is similar to
\[
  \left(
    \begin{array}{c|ccc|c}
      C(f_1) & & & \multicolumn{1}{c}{} &  \\
      \cline{1-4}
      & C(f_1) & E_2 & &\\
      & & C(f_1) & E_2 &\\
      & & & C(f_1) &\\
      \cline{2-5}
      \multicolumn{1}{c}{} & & & & C(f_2)
    \end{array}
  \right),
\]
where $C(f)$ is the companion matrix of $f$ and $E_2$ is the identity matrix of order $2$.

\section{Experiments}
\label{sec:exp}

We have implemented the algorithms introduced above in the computer algebra
system Risa/Asir (\cite{nor2003}) and evaluated them. 
In \Cref{alg:ker-f(A)-lbar-generators}, it is assumed that 
minimal annihilating polynomials are pre-calculated by using the method
in our previous paper (\cite{taj-oha-ter2018}) and given.
A product $g_j(A)\e_j$ of a matrix polynomial and a vector can be computed efficiently by using Horner's rule (\cite{taj-oha-ter2018}, \cite{taj-oha-ter2018b}).
Similarly, in \Cref{alg:jordan-krylov-basis}, $f(A)$ 
can also be calculated efficiently by using an improved Horner's rule for matrix polynomials that reduces the number of matrix-matrix multiplications (\cite{taj-oha-ter2014}). 
For the reduction of the {\KrylovGS}, 
Procedure~\ref{alg:reduction-krylovGS} was used (see \Cref{remark:reduction-krylovGS}).

We have executed experiments for varying matrix sizes.
The test was carried out in the following environment: 
Apple M1 Max up to 3.2 GHz, RAM 32 GB, macOS 12.2.1, Risa/Asir version 20220309.
All data are the average of measurements from 10 experiments.

\subsection{Computing generalized eigenvectors of matrices with Jordan chains of several lengths}
\label{sec:exp-20220526}

In the experiments, matrices are given as follows. Let $f(\lambda)$ be a monic polynomial
of degree $d$ with integer coefficients. Let
\[
  A' =
  \left(
    \begin{array}{ccc|cc|cc|c|c|c}
      C(f) & E_d &  \\
      & C(f) & E_d & \\
      & & C(f) & \\
      \cline{1-5}
      & & & C(f) & E_d & \\
      & & & & C(f) & \\
      \cline{4-7}
      \multicolumn{4}{c}{} & & C(f) & E_d \\
      \multicolumn{4}{c}{} & & & C(f) \\
      \cline{6-8}
      \multicolumn{6}{c}{} & & C(f)  \\
      \cline{8-9}
      \multicolumn{7}{c}{} & & C(f) \\
      \cline{9-10}
      \multicolumn{8}{c}{} & & C(f)
    \end{array}
  \right),
\]
where $C(f)$ denotes the companion matrix of $f(\lambda)$
and $E_d$ denotes the identity matrix of order $d$.
The test matrix $A$ is calculated by applying a similarity transformation to $A'$, as follows.
Let $T_{ij}(c)$ be the matrix obtained by adding $c$ times the $i$-th row of the identity matrix to the $j$-th row.
Then, $A$ was obtained by multiplying $T_{ij}(c)$ from the left and $T_{ij}(c)^{-1}$ from the right several times until $A$ becomes dense, where $c\in\Z$ is a small random integer.
Clearly, $\chi_A(\lambda)=f(\lambda)^{10}$.
Let $\cE=\{\e_1,\e_2,\dots,\e_{10d}\}$ be the standard basis of $\K^{10d}$.
Then, the minimal annihilating polynomials are calculated as
$\pi_{A,\e_j}(\lambda)=f(\lambda)^3$ for $j=1,\dots,10d$.

\Cref{tab:test-20240615-1} shows the results of the computation without the reduction of 
the {\KrylovGS} by Procedure~\ref{alg:reduction-krylovGS}. 
Computing times are measured in seconds.
Memory usage, measured in megabytes, represents the total amount of memory
requested by the program, not necessarily the amount of memory used
at one time.
``Average $|a_{ij}|$'' and
``max $|a_{ij}|$'' denote the average and the maximum of the absolute value 
of the element of $A$, respectively.
To keep the test matrices comparable,
we fix the magnitude of the determinant 
by fixing the constant term of $f(\lambda)$.
In the experiments reported in this subsection, all test matrices satisfy
Average $|a_{ij}|<10^3$ and
$\max|a_{ij}|<10^4$, thus we expect that the difference in the test matrices does not have much effect on the computing time.

\Cref{tab:test-20240615-1-detail-annih} shows the computing time of each step of the algorithm.
The column ``$f(A)$'' corresponds to the computing time of $f(A)$;
``Alg.~\ref{alg:ker-f(A)-lbar-generators}'' corresponds to the computing time of a {\KrylovGS};
``Alg.~\ref{alg:jordan-krylov-basis}'' corresponds to the computing time of Jordan-Krylov bases;
``Alg.~\ref{alg:single-jordan-chain}'' corresponds to the computing time of Jordan chains,
calculated as shown in the example in \Cref{sec:examples}.
``$\pi_{A,j}(\lambda)$'' corresponds to the computing time of the unit minimal annihilating polynomials (it is excluded from the total computing time in \Cref{tab:test-20240615-1}, and
written with parentheses). 
We note that, throughout the experiments, $\pi_{A,j}(\lambda)$ is calculated by the probabilistic algorithm (see \Cref{rem:pseudo-minimal-annihilating-polynomial}).

Now, the calculation results were compared with those that also included the reduction of the 
{\KrylovGS} by Procedure~\ref{alg:reduction-krylovGS}. 
The calculation results are shown in \Cref{tab:test-20240615-with-reduction-krylovGS}. 
The columns under ``Without Proc.~\ref{alg:reduction-krylovGS}'' are the same results as in 
\Cref{tab:test-20240615-1}.
The columns under ``With Proc.~\ref{alg:reduction-krylovGS}'' are the results when Procedure~\ref{alg:reduction-krylovGS} was also executed.
For each category, ``Digits'' corresponds to the maximum number of decimal digits appearing in the integer coefficients of vector components during the computation.
Since our complexity evaluation is based on arithmetic complexity, it may not accurately reflect differences arising from the growth of integer coefficients during computation. 
Nevertheless, the results suggest that Procedure~\ref{alg:reduction-krylovGS} tends to slightly reduce the coefficient size. The observed difference in memory consumption may also be partially attributable to this reduction in coefficient size. 
These effects may contribute to the shorter computation times observed when Procedure~\ref{alg:reduction-krylovGS} is applied.

\Cref{tab:test-20240615-with-reduction-krylovGS-detail-annih} shows the computing time of each step of the algorithm. 
The column ``Proc.~\ref{alg:reduction-krylovGS}'' corresponds to the computing time of the reduction of the {\KrylovGS}, and the rest of the columns are the same as those in \Cref{tab:test-20240615-1-detail-annih}.

\begin{table}[t]
  \centering
  \caption{Computing time and memory usage for the case in \Cref{sec:exp-20220526} without employing Procedure~\ref{alg:reduction-krylovGS}.}
  \label{tab:test-20240615-1}
  \begin{tabular}{r|r|r|r|r|r}
    \hline
    $\deg(f)$ & $\textrm{size}(A)$ & Time (sec.) & Memory usage (MB) & Average $|a_{ij}|$ & max $|a_{ij}|$\\
    \hline%
    2 & 20 & 0.02 & 23.89 &  33.76 & 322 \\
    4 & 40 & 0.16 & 206.51 &  74.03 & 1409 \\
    6 & 60 & 0.70 & 819.13 &  43.33 & 496 \\
    8 & 80 & 1.72 & $2.51\times 10^3$ &  108.49 & 2397 \\
    10 & 100 & 3.86 & $6.31\times 10^3$ &  66.24 & 1977 \\
    12 & 120 & 11.58 & $1.59\times 10^4$ & 91.78 & 4179 \\
    14 & 140 & 20.59 & $3.10\times 10^4$ &  87.61 & 1229 \\
    16 & 160 & 40.72 & $5.83\times 10^4$ &  103.40 & 4944 \\
    18 & 180 & 90.07 & $1.36\times 10^5$ &  195.82 & 8226 \\
    20 & 200 & 151.09 & $2.13\times 10^5$ & 182.82 & 5457 \\
    \hline
  \end{tabular}
\end{table}

\begin{table}[tb]
  \centering
  \caption{Computing time of each step of the algorithm for the case in \Cref{sec:exp-20220526} without employing Procedure~\ref{alg:reduction-krylovGS}. Computing time for the unit minimal annihilating polynomials ($\pi_{A,j}(\lambda)$) is not included in the total computing time (\Cref{tab:test-20240615-1}).}
  \label{tab:test-20240615-1-detail-annih}
  \begin{tabular}{r|r|r|r|r|r|r}
    \hline
    $\deg(f)$ & $\textrm{size}(A)$ & 
    \multicolumn{1}{c|}{$f(A)$} & 
    \multicolumn{1}{c|}{($\pi_{A,j}(\lambda)$)} & 
    \multicolumn{1}{c|}{Alg.~\ref{alg:ker-f(A)-lbar-generators}} & 
    \multicolumn{1}{c|}{Alg.~\ref{alg:jordan-krylov-basis}} &
    \multicolumn{1}{c}{Alg.~\ref{alg:single-jordan-chain}} \\ 
    \hline%
    2 & 20 & $5.61\times 10^{-4}$ & ($2.48\times 10^{-3}$) & $4.48\times 10^{-5}$ & 0.02 & 0.01 \\
    4 & 40 & 0.02 & ($5.88\times 10^{-2}$) & $8.61\times 10^{-5}$ & 0.11 & 0.03 \\
    6 & 60 & 0.11 & (0.33) & $1.36\times 10^{-4}$ & 0.39 & 0.20 \\
    8 & 80 & 0.35 & (1.15) & $1.79\times 10^{-4}$ & 0.83 & 0.53 \\
    10 & 100 & 0.86 & (2.19) & $2.51\times 10^{-4}$ & 1.57 & 1.43 \\
    12 & 120 & 1.90 & (4.49) & $2.91\times 10^{-4}$ & 3.64 & 6.04 \\
    14 & 140 & 3.58 & (8.39) & $3.40\times 10^{-4}$ & 6.63 & 10.39 \\
    16 & 160 & 6.60 & (14.08) & $4.09\times 10^{-4}$ & 11.60 & 22.52 \\
    18 & 180 & 10.83 & (24.16) & $4.55\times 10^{-4}$ & 26.78 & 52.46 \\
    20 & 200 & 14.69 & (36.39) & $5.54\times 10^{-4}$ & 38.95 & 97.54 \\
    \hline
  \end{tabular}
\end{table}

\begin{table}[t]
  \centering
  \caption{Comparison of computing time, memory usage, and coefficient growth with and without Procedure~3 for the case in \Cref{sec:exp-20220526}.}
  \label{tab:test-20240615-with-reduction-krylovGS}
  \begin{tabular}{r|r|r|r|r|r|r|r}
    \hline
    \multirow{2}{*}{$\deg(f)$} & \multirow{2}{*}{$\textrm{size}(A)$} &
    \multicolumn{3}{c|}{Without Proc.~\ref{alg:reduction-krylovGS}} &
    \multicolumn{3}{c}{With Proc.~\ref{alg:reduction-krylovGS}} \\
    & & 
    \multicolumn{1}{c|}{Time} &
    \multicolumn{1}{c|}{Memory} &
    \multicolumn{1}{c|}{Digits} &
    \multicolumn{1}{c|}{Time} &
    \multicolumn{1}{c}{Memory} &
    \multicolumn{1}{c}{Digits} \\
    \hline%
    2 & 20 & 0.02 & 23.89 & 9 & 0.03 & 40.28 & 6 \\
    4 & 40 & 0.16 & 206.51 & 50 & 0.17 & 277.00 & 46 \\
    6 & 60 & 0.70 & 819.13 & 125 & 0.65 & 932.92 & 110 \\
    8 & 80 & 1.72 & $2.51\times 10^3$ & 275 & 1.40 & $2.50\times 10^3$ & 254 \\
    10 & 100 & 3.86 & $6.31\times 10^3$ & 437 & 2.96 & $5.77\times 10^3$ & 416 \\
    12 & 120 & 11.58 & $1.59\times 10^4$ & 899 & 6.14 & $1.27\times 10^4$ & 817 \\
    14 & 140 & 20.59 & $3.10\times 10^4$ & 1196 & 11.40 & $2.34\times 10^4$ & 1075 \\
    16 & 160 & 40.72 & $5.83\times 10^4$ & 1386 & 21.23 & $4.70\times 10^4$ & 1290 \\
    18 & 180 & 90.07 & $1.36\times 10^5$ & 2350 &  42.61 & $0.95\times 10^5$ & 2047 \\
    20 & 200 & 151.09 & $2.13\times 10^5$ & 2855 & 88.15 & $1.60\times 10^5$ & 2667 \\
    \hline
  \end{tabular}
\end{table}

\begin{table}[tb]
  \centering
  \caption{Computing time of each step of the algorithm for the case in \Cref{sec:exp-20220526}  employing Procedure~\ref{alg:reduction-krylovGS}. Computing time for the minimal annihilating polynomials $\pi_{A,j}(\lambda)$ is not included in the total computing time (\Cref{tab:test-20240615-with-reduction-krylovGS}).}
  \label{tab:test-20240615-with-reduction-krylovGS-detail-annih}
  \begin{tabular}{r|r|r|r|r|r|r|r}
    \hline
    $\deg(f)$ & $\textrm{size}(A)$ & 
    \multicolumn{1}{c|}{$f(A)$} & 
    \multicolumn{1}{c|}{($\pi_{A,j}(\lambda)$)} & 
    \multicolumn{1}{c|}{Alg.~\ref{alg:ker-f(A)-lbar-generators}} & 
    \multicolumn{1}{c|}{Alg.~\ref{alg:jordan-krylov-basis}} &
    \multicolumn{1}{c|}{Proc.~\ref{alg:reduction-krylovGS}} & 
    \multicolumn{1}{c}{Alg.~\ref{alg:single-jordan-chain}} \\ 
    \hline%
    2 & 20 & $6.52\times 10^{-4}$ & ($2.48\times 10^{-3}$) & $5.01\times 10^{-5}$ & 0.02 & 0.01 & $9.95\times 10^{-4}$ \\ 
    4 & 40 & 0.01 & ($5.88\times 10^{-2}$) & $8.89\times 10^{-5}$ & 0.09 & 0.04 & 0.02 \\ 
    6 & 60 & 0.08 & (0.33) & $1.25\times 10^{-4}$ & 0.31 & 0.13 & 0.14 \\ 
    8 & 80 & 0.31 & (1.15) & $1.72\times 10^{-4}$ & 0.55 & 0.24 & 0.30 \\ 
    10 & 100 & 0.68 & (2.19) & $2.15\times 10^{-4}$ & 1.11 & 0.39 & 0.77 \\ 
    12 & 120 & 1.66 & (4.49) & $3.29\times 10^{-4}$ & 2.01 & 0.61 & 1.86 \\ 
    14 & 140 & 3.18 & (8.39) & $3.52\times 10^{-4}$ & 3.55 & 0.84 & 3.83 \\ 
    16 & 160 & 6.08 & (14.08) & $3.78\times 10^{-4}$ & 6.36 & 1.16 & 7.63 \\ 
    18 & 180 & 9.83 & (24.16) & $5.12\times 10^{-4}$ & 13.08 & 1.67 & 18.02 \\ 
    20 & 200 & 15.51 & (36.39) & $5.36\times 10^{-4}$ & 24.88 & 2.20 & 45.56 \\ 
    \hline
  \end{tabular}
\end{table}

\subsection{Computing generalized eigenvectors associated to
a specific eigenfactor}
\label{sec:exp-20220714}

In the experiments, matrices are given as follows. Let $f(\lambda)$ and $g_i(\lambda)$ 
($i=1,2,3$) be monic polynomials with integer coefficients. 
Their degrees are as follows: $\deg(g_2)=2d$, and the degrees of all the other 
polynomials are equal to $d$. Let
\[
  A' =
  \left(
    \begin{array}{ccccc|cc|c|c}
      C(f) & E_d & & & \\
      & C(f) & E_d & & \\
      & & C(f) & E_d & \\
      & & & C(f) & E_d & \\
      & & & & C(f) & \\
      \cline{1-7}
      \multicolumn{4}{c}{} & & C(g_1) & E_{2d}\\
      \multicolumn{4}{c}{} & & & C(g_1) \\
      \cline{6-8}
      \multicolumn{6}{c}{} &  & C(g_2) \\
      \cline{8-9}
      \multicolumn{7}{c}{} & &  C(g_3)
    \end{array}
  \right),
\]
then, the test matrix $A$ is calculated by applying a similarity transformation on $A'$, in the same manner as in \Cref{sec:exp-20220526}.
Clearly, $\chi_A(\lambda)=f(\lambda)^{5}g_1(\lambda)^2g_2(\lambda)g_3(\lambda)$.
Let $\cE=\{\e_1,\e_2,\dots,\e_{10d}\}$ be the standard basis of $\K^{10d}$.
Then, all the minimal annihilating polynomials 
$\pi_{A,\e_j}(\lambda)$ have $f(\lambda)^5$ as a factor.
We focus on computing the generalized eigenspace associated to the roots of $f(\lambda)$.
Note that Procedure~\ref{alg:reduction-krylovGS} is employed in this experiment.

When computing the Jordan chains associated to just one eigenvalue, the computation of 
a {\KrylovGS} in \Cref{alg:ker-f(A)-lbar-generators} can be shortened 
by combining it with the computation of the minimal annihilating polynomials
(see \Cref{rem:omit}).

\Cref{tab:test-20220714-1,tab:test-20220714-2-annih} show the results of the computation.
Note that, according to \Cref{rem:omit}, 
in the column ``Time'' in \Cref{tab:test-20220714-1}, 
the computing time of \Cref{alg:ker-f(A)-lbar-generators} is excluded,
and, in \Cref{tab:test-20220714-2-annih}, the computing time of
\Cref{alg:ker-f(A)-lbar-generators} is written in parentheses.
The rest of the contents of the tables are the same as 
\Cref{tab:test-20240615-1,tab:test-20240615-1-detail-annih,tab:test-20240615-with-reduction-krylovGS,tab:test-20240615-with-reduction-krylovGS-detail-annih}.
respectively.
Note that the computing time is only for the generalized eigenvector computations associated to
the root of $f(\lambda)$.

The results show the following.
\begin{enumerate}
  \item \Cref{alg:ker-f(A)-lbar-generators}: in \Cref{sec:exp-20220526}, since the minimal 
  annihilating polynomial has only $f(\lambda)$ as a factor, the computing time is relatively short. On 
  the other hand, in \Cref{sec:exp-20220714}, the minimal annihilating polynomial has
  more factors; thus, the computing time is relatively long.
  \item \Cref{alg:jordan-krylov-basis}: in \Cref{sec:exp-20220526}, since the number of vectors in 
  the {\JKbasis} is larger, the computing time is relatively long. On the other hand, in 
  \Cref{sec:exp-20220714}, the structure of $A'$ tells us that there is just one element in  
  the {\JKbasis}, and the computing time is very short. 
  \item Procedure~\ref{alg:reduction-krylovGS}: in \Cref{sec:exp-20220526}, 
  \Cref{tab:test-20240615-with-reduction-krylovGS} shows that 
  calling Procedure~\ref{alg:reduction-krylovGS} from \Cref{alg:jordan-krylov-basis} reduces the 
  computing time of the algorithms.
  \Cref{tab:test-20240615-1-detail-annih,tab:test-20240615-with-reduction-krylovGS-detail-annih} show that
  the computing time of \Cref{alg:jordan-krylov-basis,alg:single-jordan-chain} is reduced.
  While Procedure~\ref{alg:reduction-krylovGS} may not be effective in all cases, it is expected to be effective in many instances, as demonstrated by this example.
  \item Comparison between the theoretical complexity estimate and the actual computing time 
  (see \Cref{thm:time-complexity}): 
  \begin{enumerate}
    \item Regarding the computation of $f(A)$, in the present experiments, we used naive matrix multiplication (an $O(n^3)$ algorithm) and Horner's method for matrix polynomials (requiring $O(d)$ matrix multiplications). 
    In agreement with the theoretical estimate $O(n^4)$,
    the results in \Cref{sec:exp-20220526,sec:exp-20220714} 
    suggest that the observed complexity is $\tilde{O}(n^4)$, more precisely, $O(n^4 \log(n))$.
    \item Regarding the computational complexity of the part following the computation of $f(A)$, in the examples in \Cref{sec:exp-20220526,sec:exp-20220714}, since $r \simeq n$, the theoretical estimate approaches $O(n^4)$.
    The experimental results in \Cref{sec:exp-20220526} suggest an approximately $\tilde{O}(n^5)$  growth of the running time, whereas those in \Cref{sec:exp-20220714} suggest an approximately 
    $\tilde{O}(n^4)$ growth.
  \end{enumerate}
\end{enumerate}

\begin{table}[tb]
   \centering
  \caption{Computing time and memory usage
   for the case in \Cref{sec:exp-20220714}.}
  \begin{tabular}{r|r|r|r|r|r}
    \hline
    $\deg(f)$ & $\textrm{size}(A)$ & Time (sec.) & Memory usage (MB) & Average $|a_{ij}|$ & max $|a_{ij}|$\\
    \hline%
    4 & 40 & 0.04 & 348.03 & 38.42	& 984 \\
    8 & 80 & 0.64 & $5.24\times 10^3$ &  38.91	& 1942 \\
    12 & 120 & 3.34 & $2.72\times 10^4$ & 63.28 & 4209  \\
    16 & 160 & 11.68 & $8.88\times 10^4$ &  131.31 & 10578 \\
    20 & 200 & 31.22 & $2.32\times 10^5$ & 108.55 & 6495 \\
    \hline
  \end{tabular}
  \label{tab:test-20220714-1}
\end{table}

\begin{table}[tb]
  \centering
  \caption{Computing time of each step of the algorithm
  for the case in \Cref{sec:exp-20220714}.
  Computing time for the minimal annihilating polynomials $\pi_{A,j}(\lambda)$ is not included in the total computing time (\Cref{tab:test-20220714-1}).}
  \begin{tabular}{r|r|r|r|r|r|r|r}
    \hline
    $\deg(f)$ & $\textrm{size}(A)$ &  \multicolumn{1}{c|}{$f(A)$}  &
    \multicolumn{1}{c|}{($\pi_{A,j}(\lambda)$)} & 
    \multicolumn{1}{c|}{(Alg.~\ref{alg:ker-f(A)-lbar-generators})} & 
    \multicolumn{1}{c|}{Alg.~\ref{alg:jordan-krylov-basis}} &
    \multicolumn{1}{c|}{Proc.~\ref{alg:reduction-krylovGS}} &
    \multicolumn{1}{c}{Alg.~\ref{alg:single-jordan-chain}} \\ 
    \hline%
    4 & 40 & 0.01 & (0.09) & (0.04) & 2.02$\times 10^{-5}$ & 0.01 & 0.02 \\
    8 & 80 & 0.24 & (1.38) & (0.67) & 3.24$\times 10^{-5}$ & 0.06 & 0.35 \\ 
    12 & 120 & 1.44 & (7.03) & (3.45) & 4.31$\times 10^{-5}$ & 0.05 & 1.85 \\ 
    16 & 160 & 5.29 & (23.78) & (11.62) & 5.41$\times 10^{-5}$ & 0.10 & 6.29 \\ 
    20 & 200 & 14.33 & (64.30) & (31.64) & 6.81$\times 10^{-5}$ & 0.18 & 16.71 \\ 
    \hline
  \end{tabular}
  \label{tab:test-20220714-2-annih}
\end{table}

\subsection{Comparison of performance with Maple}
\label{sec:maple}

In this section, we compare the performance of the proposed algorithm with the implementation of corresponding algorithms on the computer algebra system Maple. 
We conducted the experiments in \Cref{sec:exp-20220526,sec:exp-20220714} using Maple 2023.2 (\cite{maple2023}), the native version for Apple silicon, on the same computer as the above,
except that the operating system was macOS 26.7.1.
Since some of the computations with Maple take a long time, 
each computation was measured once in a fresh session, 
except for the characteristic polynomials (see \Cref{sec:maple-eigenvectors}).
The computing time is the CPU time, and the memory usage is the total amount of memory allocated during the computation 
(``memory used'' reported by Maple).
A computation was aborted if it did not finish within 8 hours, which is denoted by ``$>$8h'' in the tables.
Note that the results in \Cref{sec:exp-20220526} are based on the results for
``With Proc.\ \ref{alg:reduction-krylovGS}'' in \Cref{tab:test-20240615-with-reduction-krylovGS}.

\subsubsection{Comparison with the Frobenius normal form}
\label{sec:maple-Frobenius}

There are algorithms for computing generalized eigenvectors via the Frobenius normal 
form proposed by \cite{tak-yok1990} and by \cite{mor-kur2001}.
Unfortunately, we do not have the complete implementation of their algorithms; 
thus, we have compared an implementation for computing the Frobenius normal form in Maple with our algorithm in Risa/Asir.
In this experiment, we use the function ``FrobeniusForm'' in the ``LinearAlgebra'' package of Maple.
\Cref{tab:test-maple-Frobenius-5211-01} shows the results of the computation for the examples in \Cref{sec:exp-20220714} in the case that the characteristic polynomial has multiple irreducible factors.
Computing time shows the CPU time, including the time for garbage collection.
In the results with Maple, the reported memory usage is the total amount of memory requested by Maple, which is similar to the case of Risa/Asir.
The computing time of the minimal annihilating polynomials is included in our algorithm.
The memory usage of our algorithm is the memory usage during the computation of the minimal annihilating polynomials.

We see that the computing time of the Frobenius normal form is somewhat shorter than, but of the same order as, that of the minimal annihilating polynomials in our algorithm.
Regarding the computational complexity of the Frobenius normal form, it is known that the arithmetic complexity is at most $O(n^3)$ (\cite{sto1998}). 
On the other hand, for the computation of the minimal annihilating polynomial, when using a probabilistic algorithm, the complexity is $O(n^2 q^2 r)$ (see
Proposition~\ref{prop:time-complexity-minimal-annihilating-polynomials}). 
Here, since $q$ is small compared to $n$, and considering that the test matrix becomes dense through similarity transformation, it is reasonable to assume that $r$ is close to $n$. 
Thus, this complexity can be regarded as equivalent to $O(n^3)$. 
As a result, we see that the computing times of the Frobenius normal form and the minimal annihilating polynomial reflect their respective theoretical complexity estimates.

In the general case of computing a generalized eigenspace, calculating the Frobenius normal form may be more efficient than our approach, which computes a Krylov generating set via the minimal annihilating polynomial. 
However, as in the experimental example in Section~\ref{sec:exp-20220714},
when the generalized eigenspace of interest is relatively small compared to the overall matrix size and the Krylov generating set can be computed efficiently, 
our method is expected to compute the generalized eigenspace effectively.

\begin{table}[t]
  \centering
  \caption{Computing time and memory usage of the Frobenius normal form of the matrices used 
  in \Cref{sec:exp-20220714} using Maple.
  The memory utilization is measured in megabytes (MB).
  The column ``$f(A)+\mbox{Alg. \ref{alg:jordan-krylov-basis}--\ref{alg:single-jordan-chain}, \ref{alg:ker-f(A)-lbar-generators}}$'' shows the total computing time of \Cref{alg:jordan-krylov-basis},
  Procedure \ref{alg:reduction-krylovGS}, and \Cref{alg:single-jordan-chain,alg:ker-f(A)-lbar-generators}, shown in 
   \Cref{tab:test-20220714-2-annih}.
  See \Cref{sec:maple-Frobenius} for details.}
  \label{tab:test-maple-Frobenius-5211-01}
  \begin{tabular}{r|r|r|r|r|r|r}
    \hline
    \multirow{2}{*}{$\deg(f)$} & \multirow{2}{*}{$\textrm{size}(A)$} &
    \multicolumn{2}{c|}{Maple (FrobeniusForm)}  & 
    \multicolumn{2}{c}{($\pi_{A,j}(\lambda)$)} &
    \multicolumn{1}{c}{$f(A)+\mbox{Alg. \ref{alg:jordan-krylov-basis}--\ref{alg:single-jordan-chain}, \ref{alg:ker-f(A)-lbar-generators}}$}\\
    & & Time (sec.) & Memory 
    & Time (sec.) & Memory & Time (sec.)\\
    \hline%
    4 & 40 & 0.14 & 28.51 & (0.09) & 10.97 & 0.09 \\
    8 & 80 & 1.07 & $2.46\times 10^2$ & (1.38) & 162.85 & 1.31\\
    12 & 120 & 5.69 & $1.23\times 10^3$ & (7.03) & 860.56 & 6.79 \\
    16 & 160 & 16.73 & $3.58\times 10^3$ & (23.78) & $2.84\times 10^3$ & 23.29\\
    20 & 200 & 45.18 & $9.20\times 10^3$ & (64.30) & $7.45\times 10^3$ & 62.85\\
    \hline
  \end{tabular}
\end{table}

\subsubsection{Computing generalized eigenvectors}
\label{sec:maple-eigenvectors}

In this experiment, we compare the performance of the proposed algorithm with that of Maple by executing the function ``Eigenvectors'' 
in the ``LinearAlgebra'' package, which computes (generalized) eigenvectors. 
According to its source code\footnote{In Maple, the source code can be viewed via the following commands: ``\texttt{interface(verboseproc = 2): print(LinearAlgebra[Eigenvectors]);}'' (\cite{maple13-intro-prog}).}, the function calculates the characteristic polynomial when the components of the matrix are integers or rational numbers. Then, by sequentially solving the system of linear equations, it calculates the general eigenvectors in the algebraic extension field.

\Cref{tab:test-maple-322111-02,tab:test-maple-5211-01} show the results of computation. 
Computing time includes the time for calculating the characteristic polynomial.
The results of the ``Proposed method'' are those shown in 
\Cref{tab:test-20240615-with-reduction-krylovGS,tab:test-20220714-1} 
(With Proc.\ \ref{alg:reduction-krylovGS}), respectively.
We see that while computing time in Maple is long, the proposed method efficiently computes the generalized eigenspace.

For reference, \Cref{tab:test-maple-charapoly-322111-02,tab:test-maple-charapoly-5211-01} show the 
computing time of the characteristic polynomial of $A$ used in
\Cref{sec:exp-20220526,sec:exp-20220714}, respectively, by using Maple.
Note that the computing time is denoted in milliseconds (ms).
Since these computations are short, 
the computing time and memory usage are the average of 10 measurements in the same session, 
performed after a warm-up call to exclude the time for loading the library.
Although the degree of $f(\lambda)$ appearing in the experiment in 
\Cref{tab:test-maple-322111-02,tab:test-maple-5211-01} is up to 
12 and 8, respectively, 
we have measured the computing time of all the characteristic polynomials used in 
\Cref{sec:exp-20220526,sec:exp-20220714}.

\begin{table}[t]
   \centering
  \caption{Computing time and memory usage of Generalized eigenvectors for the case in \Cref{sec:exp-20220526} using Maple.
  The memory utilization is measured in megabytes (MB).
  See \Cref{sec:maple-eigenvectors} for details.}
  \label{tab:test-maple-322111-02}
  \begin{tabular}{r|r|r|r|r|r}
    \hline
    \multirow{2}{*}{$\deg(f)$} & \multirow{2}{*}{$\textrm{size}(A)$} &
    \multicolumn{2}{c|}{Maple (Eigenvectors)}  & 
    \multicolumn{2}{c}{Proposed method} \\
    & & Time (sec.) & Memory usage & Time (sec.) & Memory Usage \\
    \hline%
    2 & 20 & 0.61 & 107.95 & 0.03 & 40.28 \\
    4 & 40 & 10.80 & $2.05\times 10^3$ & 0.17 & $2.7\times 10^2$\\
    6 & 60 & 47.94 & $9.94\times 10^3$ & 0.65 & $9.3\times 10^2$\\
    8 & 80 & 301.5 & $6.46\times 10^4$ & 1.40 & $2.5\times 10^3$\\
    10 & 100 & 1175.9 & $2.60\times 10^5$ & 2.96 & $5.8\times 10^3$ \\
    12 & 120 & 5876.0 & $8.79\times 10^5$ & 6.14 & $1.27\times 10^4$\\
    \hline
  \end{tabular}
\end{table}

\begin{table}[t]
  \centering
  \caption{Computing time and memory usage of Generalized eigenvectors for the case in 
  \Cref{sec:exp-20220714} using Maple.
  The memory utilization is measured in megabytes (MB).
  See \Cref{sec:maple-eigenvectors} for details.}
  \label{tab:test-maple-5211-01}
  \begin{tabular}{r|r|r|r|r|r}
    \hline
    \multirow{2}{*}{$\deg(f)$} & \multirow{2}{*}{$\textrm{size}(A)$} &
    \multicolumn{2}{c|}{Maple (Eigenvectors)}  & 
    \multicolumn{2}{c}{Proposed method} \\
    & & Time (sec.) & Memory usage & Time (sec.) & Memory Usage \\
    \hline%
    4 & 40 & 55.88 & $1.12\times 10^4$ & 0.04 & $0.35\times 10^3$\\
    8 & 80 & 2352.26 & $3.89\times 10^5$ & 0.64 & $5.24\times 10^3$\\
    12 & 120 & $>$8h & N/A & 3.34 & $2.72\times 10^4$\\
    \hline
  \end{tabular}
\end{table}

\begin{table}[t]
   \centering
  \caption{Computing time and memory usage of the characteristic polynomials of the matrices used 
  in \Cref{sec:exp-20220526}, using Maple.
  The memory utilization is measured in megabytes (MB).
  See \Cref{sec:maple} for details.}
  \label{tab:test-maple-charapoly-322111-02}
  \begin{tabular}{r|r|r|r}
    \hline
    $\deg(f)$ & $\textrm{size}(A)$ & Time (ms) & Memory usage \\
    \hline%
    2 & 20 & 1.2 & 0.10 \\
    4 & 40 & 3.1 & 0.21 \\
    6 & 60 & 7.3 & 0.35 \\
    8 & 80 & 16.0 & 0.57 \\
    10 & 100 & 27.5 & 0.83 \\
    12 & 120 & 45.9 & 1.14 \\
    14 & 140 & 69.6 & 1.49 \\
    16 & 160 & 101.4 & 1.94 \\
    18 & 180 & 142.4 & 2.42 \\
    20 & 200 & 224.1 & 2.91 \\
    \hline
  \end{tabular}
\end{table}

\begin{table}[t]
   \centering
  \caption{Computing time and memory usage of the characteristic polynomials of the matrices used 
  in \Cref{sec:exp-20220714}, using Maple.
  The memory utilization is measured in megabytes (MB).
  See \Cref{sec:maple} for details.}
  \label{tab:test-maple-charapoly-5211-01}
  \begin{tabular}{r|r|r|r}
    \hline
    $\deg(f)$ & $\textrm{size}(A)$ & Time (ms) & Memory usage \\
    \hline%
    4 & 40 & 5.8 & 0.21 \\
    8 & 80 & 34.0 & 0.56 \\
    12 & 120 & 104.9 & 1.12 \\
    16 & 160 & 237.0 & 1.94 \\
    20 & 200 & 455.2 & 2.89 \\
    \hline
  \end{tabular}
\end{table}

\subsubsection{Computing Jordan canonical form and Jordan chains}
\label{sec:maple-Jordan}

In this experiment, we compare the performance of the proposed algorithm with that of Maple by executing the function ``JordanForm'' in the ``LinearAlgebra'' package,
which computes the Jordan canonical form with the Jordan chains.
In our experiment, we set the option ``output = `Q' '' to obtain the Jordan chains.

\Cref{tab:test-maple-Jordan-322111-02,tab:test-maple-Jordan-5211-01} show the results of the computation.
The results of ``Proposed method'' are those shown in 
\Cref{tab:test-20240615-with-reduction-krylovGS,tab:test-20220714-1} 
(With Proc.\ \ref{alg:reduction-krylovGS}), respectively.
We see that, in Maple, computing times are very long even for small-sized inputs, whereas the proposed method computes the Jordan chains efficiently.
For the example in \Cref{sec:exp-20220714}, Maple did not finish the computation within 8 hours even for $\deg(f)=4$ ($\textrm{size}(A)=40$).
In contrast, the proposed method achieves computational efficiency by restricting the eigenvalues of interest to the roots of $f(\lambda)$.

\begin{table}[t]
   \centering
  \caption{Computing time and memory usage of Jordan chains for the case in \Cref{sec:exp-20220526} using Maple.
  The memory utilization is measured in megabytes (MB).
  See \Cref{sec:maple-Jordan} for details.}
  \label{tab:test-maple-Jordan-322111-02}
  \begin{tabular}{r|r|r|r|r|r}
    \hline
    \multirow{2}{*}{$\deg(f)$} & \multirow{2}{*}{$\textrm{size}(A)$} &
    \multicolumn{2}{c|}{Maple (JordanForm)}  & 
    \multicolumn{2}{c}{Proposed method} \\
    & & Time (sec.) & Memory usage & Time (sec.) & Memory Usage \\
    \hline%
    2 & 20 & 0.21 & 35.13 & 0.03 & 40.28 \\
    4 & 40 & 468.5 & $2.36\times 10^4$ & 0.17 & $0.28\times 10^3$\\
    6 & 60 & $>$8h & N/A & 0.65 & $0.93\times 10^3$\\
    \hline
  \end{tabular}
\end{table}

\begin{table}[t]
   \centering
  \caption{Computing time and memory usage of Jordan chains for the case in \Cref{sec:exp-20220714} using Maple.
  The memory utilization is measured in megabytes (MB).
  See \Cref{sec:maple-Jordan} for details.}
  \label{tab:test-maple-Jordan-5211-01}
  \begin{tabular}{r|r|r|r|r|r}
    \hline
    \multirow{2}{*}{$\deg(f)$} & \multirow{2}{*}{$\textrm{size}(A)$} &
    \multicolumn{2}{c|}{Maple (JordanForm)}  & 
    \multicolumn{2}{c}{Proposed method} \\
    & & Time (sec.) & Memory usage & Time (sec.) & Memory Usage \\
    \hline%
    4 & 40 & $>$8h & N/A & 0.04 & $0.35\times 10^3$ \\
    8 & 80 &  $>$8h & N/A & 0.64 & $5.24\times 10^3$\\
    \hline
  \end{tabular}
\end{table}

\section{Concluding remarks}

In this paper, we have proposed an exact algorithm for computing
generalized eigenspaces of matrices of integers or rational numbers
by exact computation. The resulting algorithm computes generalized eigenspaces
and the structure of Jordan chains 
by computing a {\JKbasis} of $\ker f(A)^{\lbar}$ using {\JKelim}.
In {\JKelim}, minimal annihilating polynomials are effectively used
for computing a {\KrylovGS}.

A feature of the present method is that 
it computes generalized eigenvectors in decreasing order 
of their ranks, which is the opposite of how they are computed in a conventional manner.
In the conventional method, such as the one shown in Maple's built-in function 
(see \Cref{sec:maple-eigenvectors}), 
generalized eigenvectors are computed in 
increasing order, and it is hard to construct Jordan chains in general.
In contrast, in the proposed method, Jordan chains of eigenvectors
are directly computed from {\JKbasis} without solving 
generalized eigenequations.

Another feature of the present method is that 
the cost of computation can be reduced by computing
just a part of the generalized eigenspace associated to a specific eigenvalue.
In conventional methods, even if 
the generalized eigenspace of our interest is relatively small compared to the whole
generalized eigenspace,
the whole matrix gets transformed
for solving the system of linear equations or 
computing canonical forms of the matrix, 
which may make the computation inefficient.
On the other hand, the present method concentrates the computation on 
$\ker f(A)^{\lbar}$ using {\JKelim}, which means that 
a smaller part of the generalized eigenspace can be computed with a smaller amount of computation
in an effective way.

The experimental results show that the computation of the minimal annihilating polynomials takes a relatively large portion of the total time required to compute the generalized eigenspace. In the experiments of \Cref{sec:exp-20220526}, it was comparable to the computing time of \Cref{alg:jordan-krylov-basis} 
(see \Cref{tab:test-20240615-1-detail-annih,tab:test-20240615-with-reduction-krylovGS-detail-annih}), while in the experiments of \Cref{sec:exp-20220714}, it dominated the overall computing time (see 
\Cref{tab:test-20220714-2-annih}).
Two points can be made regarding the improvement of computational efficiency, including that of the minimal annihilating polynomial. The first is an efficiency gain that depends on the structure of the given problem. In the experiments, 
the similarity transformations performed during the definition of the input matrices resulted in an increase in the number of bases for which $f(\lambda)$ is a factor of the minimal annihilating polynomial (denoted by $r$ in \cref{eq:rt}). However, depending on the specific structure of the given problem, the value of $r$ may be relatively small. In such cases, it is expected that the computational efficiency of the minimal annihilating polynomial will improve.
Of course, finding an application in which such an effect can be realized remains one of the future challenges.
The second point is an efficiency improvement through implementation. As pointed out in our previous paper (\cite{taj-oha-ter2018}), the computation of the minimal annihilating polynomial includes parts that can be parallelized. Therefore, an implementation using parallel computation is expected to improve computational efficiency. This is also one of our future challenges.

As a further extension of the proposed method, it would be beneficial to apply the proposed method to matrices with polynomial entries, 
for efficiently computing the generalized eigenspace in the algebraic extension field of the rational function field.
Alternative methods for constructing Krylov generating sets, including randomized, incremental, and parallel approaches, are also interesting directions for future work.

\begin{acks}
The authors thank the anonymous reviewers for their careful reading of the manuscript and their helpful comments and suggestions, which greatly improved the paper.
\end{acks}

\def\cprime{$'$}

\end{document}